\documentclass{amsart}
\usepackage{amssymb,stmaryrd}
\usepackage{amsfonts}
\usepackage{amstext}
\usepackage{geometry}                % See geometry.pdf to learn the layout options. There are lots.
\usepackage{graphicx}    %% packages
\usepackage{amssymb,amsmath,epsfig}
\usepackage{scalefnt}
\usepackage{MnSymbol}
\usepackage[all]{xypic}
\usepackage{slashed}
\usepackage{color}
\usepackage{extarrows}
\usepackage{stmaryrd}
\usepackage{lettrine}
\usepackage{booktabs}
\usepackage{paralist}
\usepackage{subfig}
\newtheorem{lemma}{Lemma}[section] 
\newtheorem{proposition}[lemma]{Proposition}
\newtheorem{corollary}[lemma]{Corollary}

\newtheorem{theorem}[lemma]{Theorem}

\newtheorem{definition}[lemma]{Definition}
\renewcommand{\imath}{\mathrm{i}}

\newcommand*{\doublenabla}{%
  \nabla\mkern-12mu\nabla
}   %EJB

\newcommand{\C}{\mathbb{C}}
\newcommand{\R}{\mathbb{R}}

\newcommand{\Z}{\mathbb{Z}}

\newcommand{\Hom}{\mathrm{Hom}}

\newcommand{\del}{\partial}

\newcommand{\ev}{\mathrm{ev}}

\newcommand{\extd}{\mathrm{d}}

\newcommand{\eps}{{\epsilon}}
\newcommand{\tens}{\mathop{{\otimes}}}

\newcommand{\id}{\mathrm{id}}

\renewcommand{\>}{\rangle}

\newcommand{\cX}{\mathfrak{X}}

\begin{document}

\title{Quantum geodesic flows and curvature}

%\thanks{}

\author{Edwin Beggs and Shahn Majid}
\address{Queen Mary University of London\\
School of Mathematical Sciences, Mile End Rd, London E1 4NS, UK; Department of Mathematics, Bay Campus, Swansea University, SA1 8EN, UK}

\email{e.j.beggs@swansea.ac.uk,   s.majid@qmul.ac.uk}

\maketitle

\begin{abstract}We study  geodesics flows on curved quantum Riemannian geometries using a recent formulation in terms of bimodule connections and completely positive maps. We complete this formalism with a canonical $*$ operation on noncommutative vector fields. We show on a classical manifold how the Ricci tensor arises naturally in our approach as a term in the convective derivative of the divergence of the geodesic velocity field, and use this to propose a similar object in the noncommutative case. Examples include quantum geodesic flows on the algebra of 
 $2 \times 2$ matrices,  fuzzy spheres and the $q$-sphere. \end{abstract}

\section{Introduction}

Noncommutative geometry is the idea that we can extend geometric concepts to the case where `coordinate algebras' are noncommutative. On the physics side, a motivation is the {\em quantum spacetime hypothesis} that spacetime is better modelled by noncommutative coordinates due to quantum gravity effects. This was speculated upon at various points since the early days of quantum mechanics\cite{Sny} but in a modern era specific proposals for models appeared in \cite{Ma:pla,MaRue,Hoo,DFR} among others, but without the machinery of quantum Riemannian geometry as now available in a constructive form\cite{BegMa}. Rather, such models were dictated by ideas from quantum groups such as Hopf algebra duality and born reciprocity\cite{Ma:pla}, quantum group symmetry\cite{MaRue} and classical symmetry\cite{Hoo,DFR}. By now, a great many models are known with quantum metrics and quantum Riemannian curvature, including black hole models and finite models of quantum gravity, e.g. \cite{BegMa:gra,Ma:sq,ArgMa,LirMa} . The constructive formalism used in these works follows a `ground up' approach where we start with the `coordinate algebra' $A$, define a differential structure in terms of a differential graded algebra $(\Omega,\extd)$ of `differential forms', then a quantum metric $g\in \Omega^1\tens_A\Omega^1$, then a quantum Levi-Civita or other connection $\nabla:\Omega^1\to \Omega^1\tens_A\Omega^1$, its curvature $R_\nabla:\Omega^1\to \Omega^2\tens_A\Omega^1$, etc. There is also a practical `working definition' of Ricci and in some cases Einstein tensors  which sometimes produces reasonable results in the sense of vanishing divergence with respect to $\nabla$ but which is not canonical. One of our goals in the present paper is to come at the Ricci tensor from another angle for this reason. There is also a well-known and more sophisticated Connes approach to noncommutative geometry coming out of cyclic cohomology and K-theory\cite{Con} and with examples such as the noncommutative torus, as well as physical applications including ideas for the standard model of particle physics\cite{ConMar}. This particularly makes use of the notion of a `spectral triple' or abstract Dirac operator to implicitly encode the quantum geometry. Sometimes Connes spectral triples can be realised within the constructive quantum Riemannian geometry\cite{BegMa:spe,LirMa2}, i.e. there is a useful intersection between these approaches. 

Working in the constructive formalism as in \cite{BegMa}, we continue in the present work to explore the recently introduced notion of a `quantum geodesic'\cite{Beg:geo,BegMa:geo} and particularly how it interacts with curvature. Until now, only quantum geodesics on flat examples were worked out, such as the equilateral triangle\cite{Beg:geo} and the Heisenberg algebra with a flat linear connection\cite{BegMa:geo}. We now study quantum geodesic flows on known quantum Riemannian geometries on $M_2(\C)$ and the fuzzy sphere,  and we write down but do not explicitly solve for quantum geodesics on the $q$-sphere $\C_q[S^2]$.  A parallel work \cite{LiuMa} covers flat $\lambda$-Minkowski space from the point of view of physical predictions and also a curved (but not quantum-Levi Civita) connection on the noncommutative torus. 

In studying our models, we are led to considerably improve the quantum geodesic formalism itself, solving two  fundamental problems in the original work \cite{Beg:geo}. The first is about the $*$-operation needed for a unitary theory over $\C$, which previously was proposed via some requirements but without a canonical choice. This is now rectified in Theorem~\ref{thmXstar} under the assumption that the positive linear functional $\int:A\to \C$ needed for the theory is a twisted trace $\int ab=\int\varsigma(b)a$ for an algebra automorphism $\varsigma$. This is a common enough situation in noncommutative geometry. The other issue that we address is an auxiliary braid condition in \cite{Beg:geo,BegMa:geo} which is empty in the classical case but which turns out in our examples to be too restrictive. In Corollary~\ref{corXreal}, we now understand this condition as sufficient but not necessary for compatibility of the geodesic evolution with our canonical $*$-structure and propose a weaker `improved auxiliary  condition' to replace it. 

The formalism of quantum geodesics will be outlined in the preliminaries Section~\ref{secpre}, with the new general results in Section~\ref{secstar}. The formalism itself involves a radically new way of thinking about geodesics, even classically, and  Section~\ref{seccla} provides new results for the formalism applied in the classical case to a Riemannian manifold. Qualitatively speaking, the idea of quantum geodesics is not to follow one particle at a time but to think of a `fluid of particles' each moving along geodesics. In reality what would be the density will actually be a quantum-mechanical probability density $\rho=|\psi|^2$ for a complex wave function over the manifold. However, keeping the fluid analogy in mind,  the tangent vectors to all the geodesic motions can be viewed together as a {\em velocity vector field} $X$ on the manifold. The key point in \cite{Beg:geo} is that this vector field is characterised by a {\em geodesic velocity equation} as autoparallel with respect to a linear connection and can be solved for first, independently of the particles themselves. A further equation, which we call the {\em amplitude flow equation} is a Schroedinger-like equation on the wave function $\psi$ relative to $X$ that evolves it in such a way that $\rho$ corresponds to the particles at each point of the manifold having velocity given by $X$ at that point. Thus, we rip apart the usual notion of a geodesic into the particle locations and the particle velocities and then we put them back in reverse order first solving for $X$ and then for $\psi$.  This is a radically different approach and some detailed examples even on flat $\R^n$ are provided in \cite{LiuMa} to help with this conceptual transition.  Our new result in Section~\ref{seccla} in this context is that if $X$ is a geodesic velocity field then the convective derivative $D/D t$ along $X$ (defined as usual in fluid dynamics) obeys
\[ {D\over Dt}{\rm div}(X)=-(\nabla_\mu X^\nu)(\nabla_\nu X^\mu) - R_{\mu\nu}X^\mu X^\nu\]
where $X$ has components $X^\mu$ in local coordinates and $R_{\mu\nu}$ is the Ricci tensor associated to the connection. The latter would normally be the Levi-Civita connection but in fact the theory at this level does not require a metric, just a linear connection. This gives  a striking new way of thinking about the familiar connection in GR between the Ricci tensor and the change of position of nearby objects in geodesic motion. We also show how the classical role of the Riemann curvature as controlling geodesic deviation looks in this new language. In Section~\ref{secR}, we propose noncommutative versions of the two terms displayed above as $-F(X)$ and $-R(X)$ respectively, with the split suggested by good properties with respect to $*$. We see in the noncommutative examples how the `Ricci quadratic form' $R(X)$ compares with the naive Ricci tensor in the current non-canonical approach. We do not believe this to be the last word on the topic, but it can be viewed as a first  look at the problem from a fresh angle. The $2\times 2$ quantum matrices are treated in Section~\ref{secmat}, the fuzzy sphere in Section~\ref{secfuz} and the $q$-sphere in Section~\ref{secqsph}. By fuzzy sphere, we mean the standard quantisation of a coadjoint orbit in $su_2^*$ as a quotient of $U(su_2)$ by a fixed value of the quadratic Casimir, equipped now with the 3-dimensional but rotationally invariant differential calculus\cite[Example~1.46]{BegMa} and quantum Riemannian geometry from \cite{LirMa}. But the theory also applies to its finite-dimensional matrix algebra quotients which are also of interest\cite{Str,Var,Mad}. By $q$-sphere, we mean the base of the $q$-Hopf fibration as a subalgebra of the standard quantum group $\C_q[SU_2]$ in the theory of quantum principal bundles and with the quantum Riemannian geometry introduced in \cite{Ma:spi}. It is a member of the more general 2-parameter Podl\`es spheres\cite{Pod} but the only one for which the quantum Riemannian geometry has been explored. The paper has some concluding remarks in Section~\ref{secrem} with a discussion of directions for further work. 
 
\section{Preliminaries}\label{secpre}

Here we recall in more detail what a quantum geodesic is as proposed in\cite{Beg:geo} and studied further in \cite{BegMa:geo,LiuMa}. We explain the algebraic point of view which works even when the `coordinate algebra' of the spacetime is a possibly noncommutative unital algebra $A$. 

\subsection{Quantum Levi-Civita connections} \label{secqlc}

We suppose a differential structure in the form of an $A$-bimodule $\Omega^1$ of differential forms equipped with a map $\extd:A\to \Omega^1$ obeying the Leibniz rule 
\[ \extd(ab)=a\extd b+ (\extd a)b\]
and such that $\Omega^1$ is spanned by elements of the form $a\extd b$ for $a,b\in A$. This can always be extended to a full differential graded `exterior algebra' though not uniquely (there is a unique maximal one). In the $*$-algebra setting, we say we have a $*$-differential structure if $*$ extends to $\Omega$ (or at least $\Omega^1$) as a graded-involution (i.e., with an extra minus sign on swapping odd degrees) and commutes with $\extd$. A full formalism of quantum Riemannian geometry in this setting can be found in \cite{BegMa}. In particular, a metric means for us an element $g\in\Omega^1\tens_A\Omega^1$ which is invertible in the sense of a bimodule map $(\ ,\ ):\Omega^1\tens_A\Omega^1 \to A$ obeying the usual requirements as inverse to $g$. This forces $g$ in fact to be central. Then a QLC or quantum Levi-Civita connection is a bimodule connection $\nabla,\sigma$ on $\Omega^1$ which is metric compatible and torsion free in the sense in the sense 
\begin{equation}\label{gtor} \nabla g:=((\id\tens\sigma)(\nabla\tens\id)+ \id\tens\nabla)g=0, \quad  T_\nabla:=\wedge\nabla+\extd=0.\end{equation}
Here we prefer right bimodule connections $\nabla:\Omega^1\to \Omega^1\tens_A\Omega^1$ characterised by
\begin{equation}\label{nablaleib} \nabla(\omega.a)=\omega\tens\extd a+ (\nabla\omega).a,\quad \nabla(a.\omega)=\sigma(\extd a\tens\omega)+a\nabla\omega,\end{equation}
where the `generalised braiding' $\sigma:\Omega^1\tens_A\Omega^1\to \Omega^1\tens_A\Omega^1$ is assumed to exist and is uniquely determined by the second equation. Connections with the first, usual, Leibniz rule are standard while bimodule connections with the further rule from the other side appeared in \cite{DVMic,Mou}. 

There is an analogous theory of left bimodule connections with left and right swapped. In this paper, as in \cite{BegMa:geo}, we mostly prefer right bimodule connections, but we note that in the context where the generalised braiding is invertible we can go freely back and forth between a right $\nabla,\sigma$ as above and an equivalent left bimodule connection $\nabla^L,\sigma_L$ according to
\begin{equation}\label{nabnabL} 
\nabla^L=\sigma^{-1}\nabla,\quad\sigma_L=\sigma^{-1},\quad \nabla=\sigma_L^{-1}\nabla^L,\quad \sigma=\sigma_L^{-1}.\end{equation}
It will be useful to use both versions related in this way. We will also have recourse to a space of `left quantum vector fields' defined as the $A$-bimodule of left $A$-module maps
\[ \cX={}_A\hom(\Omega^1,A),\quad (a.X.b)(\omega)= (X(\omega.a))b\]
for all $\omega\in \Omega^1,a,b\in A$ and $X\in\cX$. Moreover, if $\nabla^L,\sigma_L$ is a left bimodule connection on $\Omega^1$ and the latter is finitely generated projective (f.g.p.) as a left $A$-module then $\cX $ canonically acquires a right bimodule connection $\nabla_\cX: \cX \to \cX \tens_A \Omega^1$ with $\sigma_\cX:\Omega^1\tens_A\cX
\to \cX \tens_A\Omega^1 $. Here, $\nabla_\cX$ obeys Leibniz rules as in (\ref{nablaleib}) but with $\omega\in \Omega^1$ replaced by $X\in \cX $ and $\sigma$ replaced by $\sigma_\cX$. We refer to \cite[Prop. 3.32]{BegMa} for details, but the key idea is that $\nabla_{\cX}$ is characterised as preserving the evaulation map  $\ev:\Omega^1\tens_A \cX \to A$, i.e.
\[
\extd\,\ev(\omega\tens X) = (\id\tens\ev)(\nabla^L(\omega)\tens X)+(\ev\tens \id)(\omega\tens \nabla_\cX(X))\ .
\]
for all $\omega\in\Omega^1$ and $X\in \cX $. The map $\sigma_\cX$ is uniquely determined from $\nabla_{\cX}$ but likewise obtained by dualisation of $\sigma_L$, see \cite[Prop. 3.80]{BegMa}. %There is a different space $\cX^R=\hom_A(\Omega^1,A)$ of right vector fields i.e., right $A$-module maps again forming an $A$-bimodule, but we will not need it here. 

\subsection{$A$-$B$ bimodule connections and geodesic bimodules} 

So far we have discussed only linear connections $\nabla$ on $\Omega^1$ and $\nabla_\cX$ on $\cX$, but similar notions apply for (right) bimodule connections $\nabla_E:E\to E\tens_A\Omega^1$ on any $A$-bimodule $E$. Here, $E$ is  thought of as the space of sections of a vector bundle if $A$ is thought of as the coordinate algebra on the base. The generalised braiding $\sigma_E:\Omega^1\tens_A E\to E\tens_A\Omega^1$ is a bimodule map and the two Leibniz rules follow the same form as (\ref{nablaleib}). One has a notion of tensor product of $A$-bimodules with bimodule connection following the same form as $\nabla g$ in (\ref{gtor}).  Details are in \cite{BegMa} but omitted since we will need in fact a relative version, of which this is just the diagonal case. 

Thus, we will need the notion of an $A$-$B$ bimodule connection $\nabla_E$ on an $A$-$B$ bimodule $E$, where $B,\Omega^1_B$ is another algebra with differential calculus\cite[Def.~4.69]{BegMa}. This is a novel concept even in the classical case. For the right handed theory, $\nabla_E: E\to E\tens_B \Omega^1_B$ and
\[  \nabla_E(eb)=e\tens\extd b+(\nabla_E a)b,\quad \nabla_E(ae)=\sigma_E(\extd a\tens e)+a\nabla_E e   \]
for all $e\in E$, $a\in A$ and $b\in B$, for some $A$-$B$ bimodule map $\sigma_E: \Omega^1\tens_A E\to E\tens_B\Omega^1_B$. Moreover, if $E,F$ are respectively an $A$-$B$ bimodule with bimodule connection and a $B$-$C$ bimodule with bimodule connection (there are now potentially three algebras $A,B,C$ with differential calculi) then $E\tens_B F$ is an $A$-$C$ bimodule with bimodule connection by
\[ \nabla_{E\tens F}=(\id\tens \sigma_{F})(\nabla_E\tens\id) +\id\tens\nabla_F,\quad \sigma_{E\tens F}=(\id\tens\sigma_F)(\sigma_E\tens\id).\]
 In particular,  given an $A$-$B$ bimodule with bimodule connection, both domain $ \Omega^1\tens_A E$ and codomain $E\tens_B\Omega^1_B$ of $\sigma_E$ acquire tensor product $A$-$B$ bimodule connections given one on $E$, a bimodule connection $\nabla$ on $\Omega^1$  (it does not have to be a QLC), and a bimodule connection $\nabla_B$ on $\Omega^1_B$.  

The other ingredient we need is that if  $E,F$ are $A$-$B$ bimodules with bimodule connections then the set of $A$-$B$ bimodule maps $\phi:E\to F$ acquires a `covariant derivative'
%becomes an $A$-$B$ bimodule with bimodule connection. Here, the bimodule structures and the covariant derivative are
%\[ (a.\phi)(e)=\phi(a.e),\quad (\phi.b)(e)=\phi(e).b,\quad 
\begin{align} \label{hets}
\doublenabla(\phi)=\nabla_F\phi-(\phi\tens\id)\nabla_E:  E\to  F\tens_B\Omega^1_B\ . 
\end{align}
which is easily seen to be a right $B$-module map. This is more familiar in the diagonal case, where classically  it has the meaning of the covariant derivative of $\phi$ viewed as an element of the dual of $E$ tensor with $F$. We then define a  {\em strict geodesic differential bimodule} as an $A$-$B$ bimodule with bimodule connection $\nabla_E,\sigma_E$ such that  $\doublenabla(\sigma_E)=0$. We can weaken this condition.

\begin{lemma}\label{vyuid} (1) Let $E$ be an $A$-$B$ bimodule with bimodule connection, $\nabla,\sigma,\nabla_B,\sigma_B$  (right) bimodule connections on $\Omega^1,\Omega^1_B$ respectively. Then the obstruction to $\doublenabla(\sigma_E)$ being a left $A$-module (and hence $A$-$B$ bimodule) map is the mixed braid relations:
\begin{align*}
\doublenabla(\sigma_E)&(a\,\omega\tens e)-a\,\doublenabla(\sigma_E)(\omega\tens e)\\
 &= \big((\id\tens\sigma_B)(\sigma_E\tens\id)(\id\tens\sigma_E) -(\sigma_E\tens\id)(\id\tens\sigma_E)(\sigma\tens\id)
\big)(\extd a\tens\omega\tens e)
\end{align*}
for all $a\in A, \omega\in\Omega^1, e\in E$. 

(2) Suppose further that $\sigma$ is invertible and $E$ also has a (possibly unrelated) {\em left} $A$-$B$ bimodule connection $\hat\nabla_E,\hat\sigma_E$. Let $\alpha:\Omega^1\tens_A E\to E\tens_B\Omega^1_B\tens_B\Omega^1_B$ be
\begin{align*}
\alpha :=\doublenabla(\sigma_E)- \big((\id\tens\sigma_B)(\sigma_E\tens\id)(\id\tens\sigma_E) -(\sigma_E\tens\id)(\id\tens\sigma_E)(\sigma\tens\id)  \big)  \hat\nabla_{\Omega^1\tens E},
\end{align*}
where we use the left tensor product connection on $\Omega^1\tens_A E$ with $\nabla^L=\sigma^{-1}\nabla$  the associated left connection on $\Omega^1$. This is a left $A$-module map and the obstruction to being a right $B$-module (and hence $A$-$B$ bimodule) map is
\begin{align*}
\alpha&(\omega\tens e.b)-\alpha(\omega\tens e).b \\
&= - \Big( \big((\id\tens\sigma_B)(\sigma_E\tens\id)(\id\tens\sigma_E) -(\sigma_E\tens\id)(\id\tens\sigma_E)(\sigma\tens\id)  \big)    (\sigma^{-1}\tens\id) (\id\tens\hat\sigma_E)\Big)(\omega\tens e \tens\extd b)
\end{align*}
for all $b\in B, \omega\in\Omega^1, e\in E$. 
\end{lemma}
\proof We begin with (\ref{hets}) and use the a result from \cite[p. 302]{BegMa}, but in the right handed version, to give
\[
\doublenabla(\phi)(a\, e)-a\,\doublenabla(\phi)( e) = \big(\sigma_F(\id\tens\phi)-(\phi\tens\id)\sigma_E
\big)(\extd a\tens e)\ .
\]
If we set $\phi$ to be $\sigma_E$ and use the appropriate $\sigma$ for its domain and codomain then we get the first displayed equation.
From this it follows that $\alpha$ is a left $A$-module map, given the left Leibniz rule. The last equation then follows because 
$\doublenabla(\sigma_E)$ is necessarily a right module map, while $\hat\nabla_{\Omega^1\tens E}$ is a left bimodule connection requiring 
$(\sigma^{-1}\tens\id) (\id\tens\hat\sigma_E)$ from its generalised braiding. 
 \endproof

Part (1) of the lemma says that a slight generalisation of a strict geodesic bimodule, namely to just require that $\doublenabla(\sigma_E)$ be a bimodule map, is equivalent to the mixed braid relation between $\sigma_B,\sigma_E,\sigma$. This braid relation appeared  as an `auxiliary condition' in the analysis of the $\doublenabla(\sigma_E)=0$ case  in previous work\cite{Beg:geo,BegMa:geo}. In this slightly more general case, $\alpha:=\doublenabla(\sigma_E)$ is a bimodule map and has the interpretation of an external driving force, but we still have the  auxiliary braid condition which turns out to be too restrictive for key examples of interest in this paper. 

We therefore have to drop that $\doublenabla(\sigma_E)$ is a bimodule map. Part (2) of the lemma says that we can then modify $\alpha$ as stated to potentially still obtain a bimodule map if we have the additional data of a second connection $\hat\nabla_E,\hat\sigma_E$ subject to the weaker braid condition stated.  For example, $\hat\sigma_E=0$ would automatically ensure that $\alpha$ is a bimodule map, again interpreted as an external driving force. We refer to this situation where the braid relation is not entailed as a {\em nonstrict geodesic bimodule} with or without external force $\alpha$. Being a bimodule map, it is only then natural to set $\alpha=0$ if we want. Note that $\sigma_E,\hat\sigma_E$ have no reason to be invertible when $A\ne B$, as they map to very different spaces.  

\subsection{Geodesic velocity field equations } Having prepared the algebraic background, we now see how these ideas relate to geodesic flows. Here and for the rest of the paper, we focus on the case $E=A\tens B$ with its canonical $A$-$B$ bimodule structure. In this case
\[ \hat\nabla_E(a\tens b)=\extd a\tens 1\tens b,\quad \hat\sigma_E=0\]
is a natural reference connection for our nonstrict geodesic bimodule, and we fix this throughout. We also identify $\Omega^1\tens_AE=\Omega^1\tens B$ in the standard way and in this case, since $\hat\nabla_E(1\tens b)=0$, we have
\[ \hat\nabla_{\Omega^1\tens E}(\omega\tens b)= \nabla^L\omega\tens b,\]
where we assume throughout that $\sigma$ is invertible so that $\nabla^L=\sigma^{-1}\nabla$ is an equivalent left connection on $\Omega^1$. Moreover, $\alpha$ being a bimodule map needs only to be specified on $\Omega^1\tens 1$, where we see that
\[ \alpha(\omega\tens 1)=\doublenabla(\sigma_E)(\omega\tens 1)- \big((\id\tens\sigma_B)(\sigma_E\tens\id)(\id\tens\sigma_E) -(\sigma_E\tens\id)(\id\tens\sigma_E)(\sigma\tens\id)  \big)( \nabla^L\omega\tens 1).\]

Of interest for geodesics, and which we also fix now for the rest of the paper, is the choice $B=C^\infty(\R)$ where $\R$ refers to the geodesic time $t$. We fix the classical calculus $\Omega^1_B=B\extd t$ with a central basis $\extd t$ and $\nabla_B\extd t=0$. The map $\sigma_B$ is the classical `flip' but $\Omega^1_B\tens_B\Omega^1_B=B\extd t\tens\extd t$ so that $\sigma_B=\id$ after these identifications. Hence in this case
\begin{equation} \label{pfu}
\alpha(\omega\tens 1)=\doublenabla(\sigma_E)(\omega\tens 1)-(\sigma_E\tens\id)(\id\tens\sigma_E)\big(( \id-\sigma)\nabla^L\omega\tens 1\big),  \end{equation}
which is the equation that we will use. Setting $\alpha=0$ will describe quantum geodesics but fixing $\alpha$ as an external bimodule map is a natural generalisation beyond this. We no longer entail the braid relations discussed above because the correction term in the expression for $\alpha$ compensates for the failure of this.

Next, for our choice of $E$, we can also take $\nabla_E$  in a standard form\cite[Prop.~5.1]{Beg:geo}
\begin{equation}\label{nablaE}\nabla_E e=(\dot e+X_t(\extd e)+e\kappa_t)\tens\extd t,\quad \sigma_E(\omega\tens e)=X_t(\omega.e)\tens\extd t\end{equation}
 given by a time-dependent left quantum vector field $X_t$ and a time dependent element $\kappa_t$ of $A$. We similarly note that an $A$-$B$  bimodule map $\alpha:\Omega^1\tens B\to A\tens B\extd t\tens \extd t$ just amounts to a fixed (not time dependent) left quantum vector field $Y:\Omega^1\to A$.  

\begin{proposition}\label{veleq} For $E=A\tens B$ in the setting above, the requirement of a nonstrict geodesic bimodule with external driving force $Y\in\cX$ reduces to the {\em geodesic velocity equations} 
\[ \dot X_t(\omega)+ [X_t,\kappa_t](\omega)+X_t(\extd X_t(\omega))-X_t(\id\tens X_t)\nabla^L\omega=Y(\omega)\]
for all $\omega\in \Omega^1$.
\end{proposition}
\begin{proof}  The calculation is essentially the same as a right-handed version of the start of the proof of \cite[Prop.~5.2]{Beg:geo} before $\doublenabla(\sigma_E)=0$ was imposed there. Namely, omitting the  $t$ on $X_t$ for brevity and identifying $\omega\tens_A 1_E=\omega\tens 1$ where $1_E=1\tens 1\in A\tens B$, we have
\begin{align*}
\doublenabla(\sigma_E)(\omega\tens 1) &= (\nabla_{E\tens\Omega^1_B}\sigma_E-(\sigma_E\tens \id)\nabla_{\Omega^1\tens E})(\omega\tens 1)\\
&= \nabla_{E\tens\Omega^1_B}(X(\omega)\tens\extd t)-(\sigma_E\tens \id)(\id\tens\sigma_E)(\nabla \omega \tens 1)-(\sigma_E\tens\id)(\omega\tens\kappa\tens\extd t)
\end{align*}
since $\nabla_{\Omega^1\tens E}=(\id\tens\sigma_E)(\nabla\tens\id)+ \id\tens\nabla_E$ and $\nabla_E 1_E=\kappa\tens\extd t$. We put this and  $\alpha(\omega\tens 1)=Y(\omega)\tens \extd t\tens\extd t$ into equation (\ref{pfu}) to obtain
\begin{align*}
Y(\omega)\tens\extd t\tens\extd t&=\doublenabla(\sigma_E)(\omega\tens 1) -(\sigma_E\tens\id)(\id\tens\sigma_E)\big(( \id-\sigma)\nabla^L\omega\tens 1\big) \\
 &= \nabla_{E\tens\Omega^1_B}(X(\omega)\tens\extd t)-(\sigma_E\tens \id)(\id\tens\sigma_E)(\nabla^L\omega\tens 1) - \sigma_E(\omega\tens\kappa)\tens\extd t\\
 &= \big( \dot X(\omega)+ X(\extd(X(\omega)))+X(\omega)\,\kappa- X(\id\tens X)\nabla^L\omega  -X(\omega\,\kappa)
 \big)\tens\extd t\tens\extd t
\end{align*}
as stated. Here $\sigma\nabla^L=\nabla$ was used in the cancellation for the second equality and we then further substituted  $\nabla_E,\sigma_E$ in terms of $X,\kappa$. 
  \end{proof}

 If  $\Omega^1$ is f.g.p.  then the  geodesic velocity equation in Proposition~\ref{veleq} is equivalent by a straightforward dualisation (following analogous steps to those in the proof of \cite[Cor.~2.3]{BegMa:geo}) to 
\begin{equation}\label{coveleq}  \dot X_t +[X_t,\kappa_t]+(\id\tens X_t)\nabla_\cX(X_t)=Y\end{equation}
 in terms of the right connection on $\cX$  dual to $\nabla^L$ explained at the end of Section~\ref{secqlc}. We see that the case $Y=0$  says in the classical limit that $X$ is autoparallel with respect to $\nabla_\cX$, i.e., the tangent vector field to a field of generalised geodesics in this sense (and in the usual sense if the connection is the Levi-Civita one). By analogous steps to the rest of the proof of \cite[Cor.~2.3]{BegMa:geo}, the auxiliary braid condition, which we have now seen is equivalent to $\doublenabla(\sigma_E)$ being a bimodule map (such as zero), can be written as 
\begin{equation}\label{oldaux} \sigma_{\cX\cX}(X\tens X)=X\tens X\end{equation}
for a certain generalised braiding $\sigma_{\cX\cX}$.  This equation will turn out to be too strong in key examples and forces us to the nonstrict case. We will, however, meet $\sigma_{\cX\cX}$ later, in Corollary~\ref{corXreal} in relation to $*$-operations.

Proceeding with $Y=0$, flows are then obtained by `integrating' $X_t$ and are characterised in a quantum mechanical  Schroedinger's equation like manner by $\nabla_E e=0$. Here $\rho=e^*e$  depends on $t$ as $e$ depends on $t$ and evolves in the classical limit as one might expect for the density of a fluid, where each particle moves with tangent vector $X_t$ evaluated at the location of the particle. (This is geodesic flow if $X_t$ obeys the geodesic velocity equation but applies generally for any flow of this type.) There is, however, one condition we need to ensure, which is that there is a positive linear functional, which we will denote $\int: A\to \C$, such that $\int \rho$ is constant in time (so can be normalised to 1). For this to happen, we need the hermitian inner product $\<e,f\>:=\int(e^*f)$ on $E$ to be preserved by $\nabla_E$, which comes down to the two {\em unitarity conditions}
 \begin{equation}\label{unitarity} \int \big(\kappa^*_t a+a\kappa_t+X_t(\extd a)\big)=0,\quad \int \big(X_t(\omega^*)-X_t(\omega)^*\big)=0\end{equation}
for all $a\in A, \omega\in\Omega^1$. In the classical limit and in the case of the Levi-Civita connection, we would take for $\int$ the Riemannian measure defined by the metric. Moreover, the first of (\ref{unitarity}) for all $a$ would amount in the classical limit to the local condition
\begin{equation}\label{kappaeq} \kappa_t+\kappa_t^*={\rm div}_\nabla X_t\end{equation}
(as reviewed in Lemma~\ref{divrot} below). In the quantum case, we can replace ${\rm div}_\nabla$ by
a divergence naturally defined by $\int$ and in this case we can set $\kappa_t$ to be ${1\over 2}$ of this divergence to similarly solve the first condition. Classically, this means choosing $\kappa_t$ real and in this case   $\nabla_E e=0$ reduces to 
\[ {D e\over D t}+ {e\over 2}{\rm div}_\nabla X_t=0\]
as expected classically for a half-density. Here, if  we have a dust of particles moving with (possibly time dependent) velocity field $X_t$ then the rate of change of any time-dependent scalar field $e$ on the manifold as computed moving with the flow is the {\em convected derivative}
\[ {D f\over D t}:=\dot f+ X(\extd f). \]
The second part of the unitarity condition (\ref{unitarity}) in the quantum case, being true for all $\omega$, determines how $*$ acts on the $X_t$ and reduces in the classical case to $X_t$ a real vector field if we take the standard measure.

\section{States and divergence of the velocity equation on a classical manifold}\label{seccla}

In order to progress the quantum geometry further in the present paper, we first revisit the classical case of our point of view with some new classical results, notably involving the Ricci tensor. We then look at what we can say in the quantum case. 

\subsection{Classical divergence, fluid dust and the Ricci tensor}

Consider an orientated Riemannian manifold $M$ with metric $g$ and its standard measure $\mu$, which is given on each coordinate chart by
\[
\int_M f\, \extd\mu = \int f \sqrt{|g|} \,\extd x^1\dots\extd x^n, 
\]
where $f$ is a function on $M$ supported in the chart and $|g|=|\det(g_{ij})|$.  We use local coordinates and index comma notation for partial derivative and semicolon for covariant derivative.  The geometric divergence of a vector field $X$ with respect to a connection is then defined as usual by ${\rm div}_\nabla X=X^{i}{}_{;i}$. We start by recalling a well-known lemma needed for the exposition.

\begin{lemma} \label{divrot} For the standard measure $\mu$ on a Riemannian manifold,  any vector field $X$ on $M$ and any  $f\in C^\infty(M)$, we have
\[
\int (X(\extd f)+ X^{i}{}_{;i}  \,  f)\,\extd \mu =0\ .
\]
\end{lemma}
\begin{proof} 
The property we want to prove is linear in $X$, so without loss of generality we choose the support of $X$ to be contained in a particular coordinate chart.  Then, by usual integration by parts,
\[
\int \sqrt{|g|} \,X^i\, f_{,i} \extd x^1\dots\extd x^n = - \int  f(   X^{i}{}_{,i}    \sqrt{|g|}  + X^{i}  \frac{\partial}{\partial x^i}  \sqrt{|g|}     ) \extd x^1\dots\extd x^n 
\]
and
\[
 (X^{i}{}_{,i}  + X^{i} ( \frac{\partial}{\partial x^i}  \sqrt{|g|} )/  \sqrt{|g|}   )  =
 (X^{i}{}_{,i}  + X^{i}  \Gamma^j{}_{ji})   =   \,X^{i}{}_{;i}.   
\]
We used the usual formula for the Christoffel symbols for the Levi-Civita connection. 
\end{proof} 

\medskip

We now examine the dynamic behaviour of this divergence when the vector field is our time-dependent $X_t$. The classical limit of the geodesic velocity equation is the autoparallel equation
\begin{align} \label{velclass}
 \frac{\partial X^i}{\partial t}  + X^s\, X^i{}_{,s}  &+ X^k\,X^j \,\Gamma^i{}_{jk}
 =0\ .
\end{align}
If we start with an intital vector field $X_0$ on $M$ and imagine that $M$ is filled with particles of dust each moving according to geodesic motion beginning with velocity $X_0$ at their starting point, then the velocity field at later proper time $t$ will become $X_t$  obeying this equation. Our new result is the following. 

\begin{proposition} \label{riccinots} If time dependent $X$ obeys (\ref{velclass}) as needed for geodesic flow then 
\begin{align*}
\frac{D\, |X|^2}{D t} =0,\quad \frac{D\, \mathrm{div}_\nabla X}{D t} =   -X^s{}_{;i}\, X^i{}_{;s}    -   X^k\,X^r \,R_{kr} 
\end{align*}
where $R_{ij}$ is the Ricci tensor.
\end{proposition}
\begin{proof} For the first part $|X|^2$ is the length squared of $X$ with respect to the metric. Then the convected derivative is
\begin{align*}
\frac{D\, |X|^2}{D t} &= \dot X^i X^j g_{ij} + X^i \dot X^j g_{ij} + X^k (X^i X^j g_{ij} )_{;k} \\
&= \dot X^i X^j g_{ij} + X^i \dot X^j g_{ij} + X^k X^i{}_{;k} X^j g_{ij} + X^k X^i X^j{}_{;k} g_{ij} =0
\end{align*}
where we have used $g_{ij;k}=0$. 
For the second part, the convected derivative is
\begin{align*}
\frac{D\, \mathrm{div}_\nabla X}{Dt}  &= \frac{\partial \, \mathrm{div}_\nabla X}{\partial t}+X^k\,\frac{\partial \, \mathrm{div}_\nabla X}{\partial x^k}
\cr
&= \dot X^{i}{}_{,i}  + \dot X^{i}  \Gamma^j{}_{ji}+ X^k X^{i}{}_{,ik}  +  X^k X^{i}{}_{,k}  \Gamma^j{}_{ji}
+  X^k X^{i}  \Gamma^j{}_{ji,k}
\end{align*}
and substituting from (\ref{velclass}) gives
\begin{align*}
\frac{D\, \mathrm{div}_\nabla X}{Dt} &= -X^s{}_{,i}\, X^i{}_{,s}  -  X^k{}_{,i}\,X^r \,\Gamma^i{}_{rk} - X^s\, X^i{}_{,si}  - X^k\,X^r{}_{,i} \,\Gamma^i{}_{rk}    -   X^k\,X^r \,\Gamma^i{}_{rk,i}     \cr
& \quad  -(X^s\, X^i{}_{,s}  + X^k\,X^r \,\Gamma^i{}_{rk}) \Gamma^j{}_{ji}+ X^k X^{i}{}_{,ik}  +  X^k X^{i}{}_{,k}  \Gamma^j{}_{ji}
+  X^k X^{i}  \Gamma^j{}_{ji,k}    \cr
&= -X^s{}_{,i}\, X^i{}_{,s}  -  X^k{}_{,i}\,X^r \,\Gamma^i{}_{rk}  - X^k\,X^r{}_{,i} \,\Gamma^i{}_{rk}  \cr
&\quad   -   X^k\,X^r \,(\Gamma^i{}_{rk,i}    
+ \Gamma^i{}_{rk} \Gamma^j{}_{ji}
-    \Gamma^j{}_{jr,k} )   \ .
\end{align*}
Now recall that the Ricci tensor is, in terms of Christoffel symbols
\[
R_{\sigma\nu}   = R^\rho{}_{\sigma\rho\nu} = \Gamma^\rho{}_{\nu\sigma,\rho}
    - \Gamma^\rho{}_{\rho\sigma,\nu}
    + \Gamma^\rho{}_{\rho\lambda}\Gamma^\lambda{}_{\nu\sigma}
    - \Gamma^\rho{}_{\nu\lambda}\Gamma^\lambda{}_{\rho\sigma}
\]
and in terms of this, we have
\begin{align*}
\frac{D\, \mathrm{div}_\nabla X}{Dt} 
&= -X^s{}_{,i}\, X^i{}_{,s}  -  X^k{}_{,i}\,X^r \,\Gamma^i{}_{rk}  - X^k\,X^r{}_{,i} \,\Gamma^i{}_{rk} -   X^k\,X^r \,\Gamma^i{}_{rj}\Gamma^j{}_{ik}   -   X^k\,X^r \,R_{kr}   \cr
&=  -X^s{}_{;i}\, X^i{}_{;s}    -   X^k\,X^r \,R_{kr} \ .
\end{align*}
\end{proof} 

\medskip
Thus, the speed along the geodesic flow is constant as expected, while the convected derivative (i.e.\ in the frame of the moving material) of the divergence is given by a dynamic `kinetic' part expressing the varying velocity field plus a geometric part consisting of the Ricci curvature as a quadratic form on the velocity field. 

\subsection{Geodesic deviation and convective derivatives}\label{secgeodev}
The classical idea of geodesic deviation imagines a given geodesic displaced an infinitesimal amount in the direction $Z$, so that instead of position $P(t)$ we have $P(t)+h\,Z(t)$ for a small parameter $h$. Thus along a particular geodesic we have the velocity $X=\dot P$ and $\delta X=\dot Z$ for the change in $X$ with respect to the parameter $h$. Then the acceleration of $Z$ along the geodesic is given by the curvature applied to $Z$ and $X$, i.e.\ the equation of geodesic deviation. 

In our case we do not consider a fixed geodesic but rather we have a time-dependent geodesic velocity field $X(t)$ which obeys the geodesic velocity equation $\dot X+\nabla_X X=0$. The above usual picture gets modified, with the perturbation now determined by a time dependent vector field $Z(t)$ in a similar role.  Note that the convective derivative of a tensor is defined in the same manner as for a function, namely ${D\over  D t}:={\extd\over \extd t}+ \nabla_X$.

\begin{proposition} Let a time dependent $X(t)$ obey the geodesic velocity equation (\ref{velclass}) and if $Z(t)$ is another time dependent vector field, let $\delta X:=\dot Z+\nabla_X Z-\nabla_Z X$. Then $X+h\delta X$ continues to obey (\ref{velclass}) to order $O(h^2)$ if and only if
\begin{align*}
\frac{D }{D t}\, \frac{D Z^i}{D t}  &  =  X^p Z^j X^k R^i{}_{kpj}.
\end{align*}
We refer to this as the geodesic deviation equation for time dependent vector fields. 
\end{proposition}
\begin{proof} Note that 
$[Z,X]  =\nabla_Z X-\nabla_X Z$
 is the Lie bracket of $X$ and $Z$ when $\nabla$ is torsion free, as in the case of a Levi-Civita connection. Then the variation of the geodesic velocity equation, i.e., requiring $\dot X+ h\dot{\delta X}+\nabla_{X+ h\delta X}(X+h\delta X)=0$ and dropping $h^2$, gives
\begin{align*}
0&=\ddot Z+\nabla_{\dot X} Z +\nabla_X \dot Z-\nabla_Z \dot X -\nabla_{\dot Z} X+\nabla_{(\dot Z+\nabla_X Z-\nabla_Z X)} X+\nabla_X (\dot Z+\nabla_X Z-\nabla_Z X) \\
&=\ddot Z+\nabla_{\dot X} Z +2\nabla_X \dot Z+\nabla_Z \nabla_X X +\nabla_{(\nabla_X Z-\nabla_Z X)} X+\nabla_X (\nabla_X Z-\nabla_Z X).
\end{align*}
Hence, the convected acceleration of $Z$ along $X$ can be computed as
\begin{align*}
\frac{D}{Dt}\frac{D}{Dt}Z=\frac{D}{Dt}(\dot Z+\nabla_X Z) &= \ddot Z+\nabla_{\dot X} Z+2\nabla_X \dot Z+\nabla_X \nabla_X Z  \\
&=  \nabla_X \nabla_Z X  -\nabla_Z \nabla_X X -\nabla_{(\nabla_X Z-\nabla_Z X)} X. 
\end{align*}
where the transition to the second equality is exactly our above equation for $\ddot Z$. We then interpret the result using
\[
([\nabla_Y,\nabla_X ]Z)^i = Y^p X^j Z^k R^i{}_{kpj}  + (\nabla_{ [Y,X]  } Z)^i. 
\]
\end{proof}
We include this result for completeness as usual geodesic deviation but using our new way of thinking classical geodesics. This also lays the groundwork for the quantum version to be addressed elsewhere. In fact, there are significant complications from the divergence of $X$ entering the quantum version of the geodesic velocity equation.

\section{Noncommutative states and divergence}\label{secstar}

This section contains new constructions at the noncommutative level, where we address the reality or $*$-involution aspects of quantum geodesic evolution and use this to develop aspects of the theory motivated by the preceding classical results about the convective derivative of the divergence. The main results are a compatibility condition for the state with respect to the quantum Riemannian geometry and a proposal for a `Ricci quadratic form'. 

\subsection{The matching of geometric and state divergences}
Here, we study the divergence of a left quantum vector field in 
$\cX ={}_A\Hom(\Omega^1,A)$. We start
with the divergence defined by an arbitrary left connection $\hat\nabla: \cX \to  \Omega^1  \tens_A  \cX $, but in the next subsection we will fix this as the left version of $\nabla_\cX$ used to define quantum geodesics. We also consider divergence defined by a functional $\phi :A\to\C$ and relate the two. 

\begin{definition} The divergence with respect to a left connection $\hat\nabla$ is defined as 
 $\mathrm{div}_{\hat\nabla} :=\ev\circ \hat\nabla: \cX \to A$.
\end{definition}
This is easily seen to obey $\mathrm{div}_{\hat\nabla}(a.X)=a.(\mathrm{div}_{\hat\nabla} X)+ X(\extd a)$ for all $a\in A, X\in\cX $. 
\begin{definition} \label{divsta}
We say that $X\in \cX $ has divergence $\mathrm{div}_\phi X\in A$  with respect to a linear functional $\phi$ if 
\[
\phi\big( a (\mathrm{div}_\phi X)    \big) +\phi(X(\extd a))= 0 
\]
for all $a\in A$.  We say that $\phi$ is {\em nondegenerate} if it 
 has the property that $\phi(ac)=0$ for all $c\in A$ implies that $a=0$.
\end{definition}

 It is easy to see that if $\phi$ is nondegenerate and $\mathrm{div}_\phi X$ exists, it is unique. Next, recall from Proposition~\ref{divrot} that classically, in the case of a Riemannian manifold, the divergences defined by the Levi-Civita connection and the standard integral are the same. We now give a sufficient condition to ensure this more generally.

\begin{lemma}
If we have a left connection
$\hat\nabla$ on $\cX $ such that
$\mathrm{div}_{\hat\nabla}$ obeys the equation in Definition~\ref{divsta}
 on a collection of left generators of $\cX $, then we can set $\mathrm{div}_\phi =\mathrm{div}_{\hat\nabla}$ on all of $\cX $. 
 \end{lemma}
 \begin{proof}  Suppose that ${\rm div}_\phi(X)$ exists for a given vector field $X$. For all $c\in A$, we have
\begin{align*}
\phi((a.X)(\extd c)) &= \phi(X(\extd c.a))=\phi(X(\extd(ac)))- \phi(X(c\extd a)) \cr
&= -\phi\big( ca (\mathrm{div}_\phi X )  +cX(\extd a)  \big)= -\phi\big( c(a \mathrm{div}_\phi X   +X(\extd a))  \big) 
\end{align*}
so ${\rm div}_\phi(a.X)$ also exists, namely ${\rm div}_\phi(a.X):=a\,\mathrm{div}_\phi X   +X(\extd a)$. The statement then follows.   
\end{proof} 

This can be stated in an alternative concise manner as the following.

\begin{proposition} \label{propdiveq}
We can set $\mathrm{div}_\phi :=\mathrm{div}_{\hat\nabla}$ on all of $\cX $ if and only if $\phi\circ{\rm div}_{\hat\nabla}=0$. 
 \end{proposition}
 \begin{proof}
  For all $a\in A$ and $X\in \cX $ we have
 \[
 \phi\circ{\rm div}_{\hat\nabla}(a.X)=\phi(\ev\circ \hat\nabla(a.X))=\phi(X(\extd a))+\phi(a\, \mathrm{div}_{\hat\nabla}(X)), 
 \]
 so if this is always zero then ${\rm div}_{\hat\nabla}(X)$ provides a valid ${\rm div}_\phi(X)$. \end{proof}

\subsection{Twisted traces and $*$-involution on vector fields}

For this subsection we suppose that $A$ is a $*$-algebra with $*$-calculus, and that $\phi:A\to \C$ is hermitian (i.e.\ $\phi(a^*)=\phi(a)^*$). We can now define a real vector field, but only relative to $\phi$.

\begin{definition} \label{realX}
We define $X\in \cX $ to be {\em real with respect to $\phi$} if 
 $\phi(X(\omega^*))=\phi(X(\omega)^*)$ for all $\omega\in\Omega^1$. \end{definition}
If $X$ real then we have
\begin{align*}
\phi(X(\extd a)) +\phi\big(  (\mathrm{div}_\phi X)^* a    \big)&= \phi(X(\extd a^{**}))+\phi\big(  (\mathrm{div}_\phi X)^* a    \big)=\phi(X(\extd a^{*})^*)+\phi\big(  (\mathrm{div}_\phi X)^* a    \big)\\
&= \phi(X(\extd a^{*}))^* +\phi\big(  (\mathrm{div}_\phi X)^* a    \big)= - \phi( a^* \mathrm{div}_\phi X  )^* +\phi\big(  (\mathrm{div}_\phi X)^* a    \big)   =0
\end{align*}
for all $a\in A$. Note also that the reality condition is just the second part of (\ref{unitarity}), and in this case we see that the first part of (\ref{unitarity}) can be satisfied by putting $\kappa_t=\tfrac12 \mathrm{div}_\phi(X_t)$. To get further, we need to make an assumption on $\phi$.

\begin{definition} \label{divr8}
We say that $\phi$ is a twisted trace if there is an algebra automorphism $\varsigma$ with
 $\phi(a  b)=\phi(\varsigma(b) a)$. 
 \end{definition}
It is easy to see that then $\phi\circ\varsigma=\phi$, and if 
 $\phi$ is furthermore nondegenerate  then $\varsigma^{-1}(a)=\varsigma(a^*)^*$.  In the following theorem, we have the assumption that $\mathrm{div}_\phi =\mathrm{div}_{\nabla}$ on all of $\cX $, and we then just use $\mathrm{div}$ for both. We will only use the notation $\mathrm{div}$ in the case when both divergences agree.

\begin{theorem} \label{thmXstar}
Suppose that  $\phi$ is a nondegenerate hermitian twisted trace with twisting map $\varsigma$, and that
$\varsigma$ extends to a map $\varsigma:\Omega^1\to\Omega^1$ by $\varsigma(a.\extd b)=\varsigma(a).\extd \varsigma(b)$. We also assume that  $\hat\nabla,\hat\sigma$ is a left bimodule connection on $\cX $ with $\mathrm{div}_\phi =\mathrm{div}_{\hat\nabla}$ on all of $\cX $ and 
given $X\in \cX $, we define $X^* \in \cX $  for $\omega\in\Omega^1$ by
\[
 X^*(\omega)=\big(\ev\,\hat\sigma(X\tens\omega^*)\big)^*\ .
\]
Then
\newline
\noindent 1)\quad $(a\,X)^* =X^*a^*$ and  $(Xa)^* =a^*\, X^*$,
\newline
\noindent 2)\quad $\phi(X^* (\omega^*))=\phi(X(\varsigma(\omega)))^*$,
\newline
\noindent 3)\quad 
 $\mathrm{div}(X^* )  =  \mathrm{div}(X)^*$,
 \newline
\noindent 4)\quad $X^{**}=X$,
\newline
\noindent 5)\quad $X\in \cX $ is real if and only if $X^*=\varsigma\circ X\circ\varsigma^{-1} $,
\newline
\noindent 6)\quad $\mathrm{div}(\varsigma\circ X\circ\varsigma^{-1})=\varsigma( \mathrm{div}(X))$,
\newline
\noindent 7)\quad if $X\in \cX $ is real then $\mathrm{div}(X)^*  = \varsigma( \mathrm{div}(X))$.
\end{theorem}
\begin{proof} 
A  brief check shows that $ X^*(a\,\omega)= a\,X^*(\omega)$ for all $a\in A$ so $X^*\in\cX $. Next
\begin{align*}
\ev(\omega\tens(aX)^*)&=\big(\ev\,\hat\sigma(aX\tens\omega^*)\big)^*=\big(\ev\,\hat\sigma(X\tens\omega^*)\big)^*a^*=
\ev(\omega\tens X^*a^*)\ ,\\
\ev(\omega\tens(Xa)^*)&=\big(\ev\,\hat\sigma(Xa\tens\omega^*)\big)^*=
\big(\ev\,\hat\sigma(X\tens(\omega a^*)^*)\big)^*=\ev(\omega a^*\tens X^*)=\ev(\omega \tens a^*\,X^*), 
\end{align*}
which checks 1).  For 2), we set $\omega=\extd a.b$ and
\begin{align*}
\phi\big(  X^* ((\extd a.b)^*)   \big) &= \phi\big(  X^*(b^*\extd a^*)   \big) = \phi\big( b^* X^* (\extd a^*)   \big) =
 \phi\big( b^* \big(\ev\,\hat\sigma(X\tens \extd a)\big)^* \big) 
 =  \phi\big( \ev\,\hat\sigma(X\tens \extd a) \, b\big)^* \ .
\end{align*}
By definition of a bimodule connection, we have 
\begin{equation}\label{eveq}
\ev\,\hat\sigma(X\tens \extd a) =\mathrm{div}(X\,a)  -   \mathrm{div}(X)\, a\end{equation}
and using this 
\begin{align*}
\phi\big( \ev\,\hat\sigma(X\tens \extd a) \, b\big)^*  &=  \phi\big( \varsigma(b)\, \ev\,\hat\sigma(X\tens \extd a) \big)^*  =  \phi\big( \varsigma(b)\, \mathrm{div}(X\,a) \big)^* -
 \phi\big( \varsigma(b)\, \mathrm{div}(X)\, a\big)^*
\\
&=   \phi\big( \varsigma(b)\, \mathrm{div}(X\,a) \big)^* -
 \phi\big( \varsigma(ab)\, \mathrm{div}(X)\big)^* \\
        &= -\phi\big( \ev(\extd \varsigma(b) \tens X\,a) \big)^*  + \phi\big( \ev(\extd \varsigma(ab )\tens X)\big)^* \\
        &= \phi\big( \ev( \varsigma( \extd(ab )- a\,\extd b)\tens X)\big)^* = \phi\big( \ev( \varsigma( \extd a.b )\tens X)\big)^* 
\end{align*}
as required. 
For part 3),
using the fact that $\phi$ of a divergence is zero and equation (\ref{eveq}),
\begin{align*}
\phi\big( X^*(\extd a)\big) &= \phi\big(  (\ev\,\hat\sigma(X\tens\extd a^*))^*\big) 
=   \phi\big(   \mathrm{div}(X\,a^*)      \big)^*   
   -     \phi\big(     \mathrm{div}(X)\,a^*    \big)^* =    -     \phi\big(     \mathrm{div}(X)\,a^*   \big)^* \\
   &=   - \phi\big(  a\,  \mathrm{div}(X)^*   \big) \ .
\end{align*}
For part 4), we have, using part 3) and equation (\ref{eveq}),
\begin{align*}
X^{ **}(\extd a^*) &= (\ev\,\hat\sigma(X^*\tens \extd a))^*
=\big(   \mathrm{div}(X^* \,a)  -   \mathrm{div}(X^* )\,a \big)^*
=  \mathrm{div}((a^*\,X)^* )^*  -  a^*\, \mathrm{div}(X^* )^* \\
&= \mathrm{div}(a^*\,X )  -  a^*\, \mathrm{div}(X ) = X(\extd a^*)
\end{align*}
and, as both $X$ and $X^{** }$ are left module maps, this means that they agree on all of $\Omega^1$. 
For part 5), first note for $X\in \cX $ we also have $\varsigma\circ X\circ\varsigma^{-1}\in\cX $. Now if
$X^*=\varsigma\circ X\circ\varsigma^{-1} $ we have
\[
\phi(X^* (\omega^*))=\phi(\varsigma\circ X\circ\varsigma^{-1} (\omega^*))=
\phi(X\circ\varsigma^{-1} (\omega^*))=\phi(X((\varsigma\omega)^*))
\]
and then part 2) shows that $X$ is real. 
Now we suppose that $X$ is real, so for all $\omega\in\Omega^1$ and $a\in A$ we have, using parts 2) and 4),
\begin{align*} 
\phi(a\,X(\omega))&=
\phi(X(a\,\omega))=\phi(X(\omega^*a^*))^*=\phi((a^*X)(\omega^*))^* =
\phi((a^*X)^*(    \varsigma \omega   )) \\
&= \phi((X^*a)(    \varsigma \omega   )) = \phi(X^* ( \varsigma \omega   ) a)= \phi(\varsigma(a)\, X^* ( \varsigma \omega   ) )
=  \phi(a\, (\varsigma^{-1}\circ X^*) ( \varsigma \omega   ) )
\end{align*}
and then nondegeneracy of $\phi$ shows that $X^*=\varsigma\circ X\circ\varsigma^{-1} $. Now for part 6),
\[
\phi(\varsigma\circ X\circ\varsigma^{-1}(\extd a))=\phi(\varsigma(X(\extd \varsigma^{-1} a)))=\phi(X(\extd \varsigma^{-1} a))=\phi( \varsigma^{-1}( a)\, \mathrm{div}(X))=
\phi( a\, \varsigma(\mathrm{div}(X)))\ .
\]
Finally, part 7) is a combination  of parts 3), 5) and 6).
\end{proof}

\begin{definition} In the context of Theorem~\ref{thmXstar}, we 
define $\phi\,\mathrm{rev}: \cX \tens_A \Omega^1\to\C$ by
\[
\phi\,\mathrm{rev}(X\tens\omega)=\phi  \, \ev(\varsigma\omega\tens X)\ .
\]
\end{definition}

Note that $\phi\,\mathrm{rev}$ is well defined because
\[
\phi\,\mathrm{rev}(Xa\tens\omega)=\phi  \big( \ev(\varsigma\omega\tens X)\,a\big) =
\phi  \big( \varsigma(a)\,\ev(\varsigma\omega\tens X) \big) = 
\phi  \big(\ev(\varsigma (a\omega)\tens X) \big)= \phi\,\mathrm{rev}(X\tens a\omega)\ ,
\]
even though the hypothetical map
$\mathrm{rev}:\cX \tens_A \Omega^1\to A$ is not well defined.
Also we have
\[
\phi\,\mathrm{rev}(X\tens\omega\,a)=\phi  \, \ev(\varsigma(\omega\,a)\tens X)
=\phi  \, \ev(\varsigma(\omega)\tens  \varsigma(a)X) = \phi\,\mathrm{rev}(\varsigma(a)X\tens\omega)\ .
\]
Thus $\phi\,\mathrm{rev}: \cX \tens_A \Omega^1\to\C$  is actually defined on the twisted cyclic tensor product $(\cX \tens_A \Omega^1)/\sim$ where $X\tens\omega\,a \, \sim\, \varsigma(a)X\tens\omega$.

\begin{proposition} In the context of Theorem~\ref{thmXstar},  for  all $X\in\cX $ and $\omega\in \Omega^1$ ,we have
\[
\phi\,\mathrm{rev}(X\tens\omega) =\phi  \, \ev(\varsigma\omega\tens X)=\phi \,\ev\,\hat\sigma(X\tens\omega),\quad 
\phi(X(\omega)^*)=\phi(X^*(\varsigma(\omega^*)))\ .
\]
\end{proposition}
\begin{proof}  First note that $\phi \,\ev\,\hat\sigma:  \cX \tens_A \Omega^1\to\C$ 
is also well defined on the twisted cyclic tensor product. It is then enough to show the equality above for $\omega=\extd a$ with $a\in A$. Then
\begin{align*}
\phi  \, \ev(\varsigma\extd a\tens X) &= \phi\big(  \varsigma(a)\mathrm{div}(X)  \big)
= \phi\big(  \mathrm{div}(X) a \big)= \big(\phi\big(  a^*\mathrm{div}(X)^* \big)\big)^*
= \big(\phi\big(  a^*\mathrm{div}(X^*) \big)\big)^* \\
&= \big(\phi\big(  \ev(\extd a^*\tens X^*) \big)\big)^* =\phi \,\ev\,\hat\sigma(X\tens\extd a)\ .
\end{align*}
Next we have, using $X^{**}=X$ and the definition of $X^*$ that
\[
\phi(X(\omega)^*)=\phi(\ev\, \hat\sigma(X^*\tens\omega^*))
\]
and then the previous result gives the second part of the statement.
\end{proof} 

We will also need the following lemma. 

\begin{lemma}  \label{vbhu}
We assume the conditions to Theorem~\ref{thmXstar} and $X$ real as in Definition~\ref{realX}. Then
\newline
\noindent 1)\quad $\phi\big(  X(\id\tens X)(\eta\tens \lambda) \big){}^* = 
\phi\big( X(\id\tens X)(\lambda^*\tens \eta^*) \big)$,
\newline
\noindent 2)\quad $\phi\big(  (\mathrm{div}(X).X)(\omega)  \big){}^* = \phi\big(  (X.\mathrm{div}(X))(\omega^*)  \big) $,
\newline
\noindent 3)\quad $\phi\big(X(\extd\, X(\omega))\big)=  - \phi\big(  (X.\mathrm{div}(X))(\omega)  \big) $
for all $\omega,\eta, \lambda\in\Omega^1$.
\end{lemma}
\begin{proof} For 1), using the assumption that $X$ is real, we have
\begin{align*}
\phi\big( & X(\id\tens X)(\eta\tens \lambda) \big)^* = \phi\big( X(\eta\,X( \lambda)) \big)^* 
=  \phi\big( X(X( \lambda)^*\,\eta^*) \big) =  \phi\big(X( \lambda)^*\, X(\eta^*) \big)   \\
&=  \phi\big(\big( X(\eta^*)^*\,X( \lambda)\big)^* \big) = 
\phi\big( X\big( X(\eta^*)^*\,\lambda\big)^* \big) = 
\phi\big( X\big(\lambda^*\, X(\eta^*)\big) \big) = 
\phi\big( X(\id\tens X)(\lambda^*\tens \eta^*) \big)\ .
\end{align*}
For 2), we use
\begin{align*}
\phi\big( & (\mathrm{div}(X).X)(\omega)  \big){}^* = \phi\big(  X(\omega\, \mathrm{div}(X))  \big){}^*
=\phi\big(  X( \mathrm{div}(X)^*\,\omega^*)  \big) =\phi\big(  \mathrm{div}(X)^*\, X(\omega^*)  \big) \\
&=\phi\big(  \varsigma(\mathrm{div}(X))\, X(\omega^*)  \big) 
=\phi\big( X(\omega^*) \, \mathrm{div}(X) \big) 
= \phi\big(  (X.\mathrm{div}(X))(\omega^*)  \big) \ .
\end{align*}
Finally,  $\phi\big(X(\extd\, X(\omega))\big)=   -  \phi\big(   X(\omega)\,\mathrm{div}(X)    \big)
= - \phi\big(  (X.\mathrm{div}(X))(\omega)  \big)$ proves 3). \end{proof}

\subsection{$*$-Preserving connections and compatibility with quantum geodesic flow}

This subsection will justify the definition of the $*$-operation in Theorem~\ref{thmXstar} and the assumptions leading up to it. 

\begin{corollary} If $\int:A\to\C$ meets the conditions for $\phi$ in Theorem~\ref{thmXstar}, and   $X\in\cX $ is real with respect to $\int$, then both halves of  the unitarity conditions (\ref{unitarity}) hold with $\kappa={\rm div}(X)/2$. 
\end{corollary}
\begin{proof} 
The first half of (\ref{unitarity}) follows from the defnition of divergence and the following equation
\[ \int{\rm div}(X)^* a=\int\varsigma({\rm div}(X))a=\int a\, {\rm div}(X)\]
for all $a\in A$ and $X\in \cX $, where we have used Theorem~\ref{thmXstar} part 7).
The second half of (\ref{unitarity}) is just the definition of $X$ real.  \end{proof}

Next, if we have a $*$-operation on an $A$-bimodule then we have a natural `reality' or $*$-preserving condition on any bimodule connection on it. We only need here the case of linear connections where the bimodule is $\Omega^1$ or its dual $\cX $. In the former case, the standard condition\cite{BegMa} can be written as
\begin{equation}\label{starpr}
\nabla^L(\omega^*)=\dagger\nabla\omega,\quad \dagger:={\rm flip}\circ(*\tens *)
\end{equation}
where the left and right connections are related as in (\ref{nabnabL}). In this context $\sigma$ is necessarily invertible with $\sigma\,\dagger\, \sigma \,\dagger=\id$. Similarly on other bimodules with $*$-structure, including $\cX$. So far, $\hat\nabla$ is any left bimodule connection on $\cX$ but henceforth we assume it is the left version of the right connection $\nabla_\cX$, 
\begin{equation}\label{hatnabla}  \hat\nabla=\sigma_\cX^{-1}\nabla_{\cX},\quad \hat\sigma=\sigma_\cX^{-1}\end{equation}
where $\nabla_\cX$ is dual to a left connection $\nabla^L$, corresponding to $\nabla$ on $\Omega^1$ as at the end of Section~\ref{secqlc}. Now we can check that Theorem~\ref{thmXstar} is fit for purpose in singling out a suitable reality condition for geodesic vector fields.

\begin{corollary}\label{corXreal}  Suppose the conditions of Theorem~\ref{thmXstar} with $\hat\nabla$ obtained from a $*$-preserving bimodule connection $\nabla$ on $\Omega^1$ as in (\ref{hatnabla}). If $X(t)$ obeys the geodesic velocity equation (\ref{coveleq}) 
and the initial vector field $X(0)$ is real 
 then   $\dot X$ is real if and only if, for all $\omega\in\Omega^1$
 and all time $t$ (we suppress the explicit $t$ dependence of $X$ for clarity),
\[
\int X(\id\tens X)(\id-\sigma)\nabla^L(\omega)=0\ .
\]
Moreover, this is true if and only if
\begin{equation}\label{auxnew}
(\id\tens\ev)(\nabla_\cX\tens\id+\id\tens\hat\nabla)(\id-\sigma_{\cX\cX}{}^{-1})(X\tens X)=0,
\end{equation}
where $\sigma_{\cX\cX}$ is defined by 
\[
\ev(\id\tens\ev\tens\id)(\sigma_L(\omega\tens\eta)\tens Y\tens Z) = \ev(\id\tens\ev\tens\id)(\omega\tens\eta\tens \sigma_{\cX\cX}(Y\tens Z)) 
\]
for all $\omega,\eta\in\Omega^1$ and $Y,Z\in \cX$. 
\end{corollary}
\begin{proof}
From the geodesic velocity equation  and then Lemma~\ref{vbhu} 3), if $X(t)$ is real then we have
at time $t$,
\begin{align*}
\int\dot X(\omega) &= \int\big(\tfrac12\big(\mathrm{div}(X)\,X-X\,\mathrm{div}(X)\big)(\omega)
- X(\extd\, X(\omega))+X(\id\tens X)\nabla\omega\big) \\
&= \int\big(\tfrac12\big(\mathrm{div}(X)\,X+X\,\mathrm{div}(X)\big)(\omega)
+X(\id\tens X)\nabla^L\omega\big).
\end{align*}
If we write $\nabla\omega=\eta\tens\lambda$ then Lemma~\ref{vbhu} 1) and 2) give
\begin{align*}
\int & \dot X(\omega){}^* 
= \int\big(\tfrac12\big(\mathrm{div}(X)\,X+X\,\mathrm{div}(X)\big)(\omega^*)
+X(\id\tens X)(\lambda^*\tens\eta^*)\big)\\
&= \int\big(\tfrac12\big(\mathrm{div}(X)\,X+X\,\mathrm{div}(X)\big)(\omega^*)
+X(\id\tens X)\sigma_L(\lambda^*\tens\eta^*)\big) -
\int X(\id\tens X)(\id-\sigma)\sigma_L(\lambda^*\tens\eta^*) \\
&= \int\big(\tfrac12\big(\mathrm{div}(X)\,X+X\,\mathrm{div}(X)\big)(\omega^*)
+X(\id\tens X)\nabla(\omega^*)\big) -
\int X(\id\tens X)(\id-\sigma) \nabla(\omega^*)  \\
&= \int\dot X(\omega^*) -
\int X(\id\tens X)(\id-\sigma) \nabla(\omega^*) 
\end{align*}
with the last integral needing to vanish for reality. We then swapped $\omega\leftrightarrow\omega^*$ for presentation of the result. Next, we set $Y\tens Z=(\id-\sigma_{\cX\cX}{}^{-1})(X\tens X)$
and the integral condition is equivalent to
\[
\int\ev(\id\tens\ev\tens\id)(\nabla^L(\omega)\tens Y\tens Z)=0 
\]
for all $\omega\in\Omega^1$. Now apply this to $a\omega$ for arbitrary $a\in A$ to get
\begin{align*}
0 &=\int a\, \ev(\id\tens\ev\tens\id)(\nabla^L(\omega)\tens Y\tens Z)+\int\ev(\id\tens\ev\tens\id)(\extd a\tens\omega\tens Y\tens Z) \\
&=\int a\, \ev(\id\tens\ev\tens\id)(\nabla^L(\omega)\tens Y\tens Z)+\int\ev(\extd a\tens Y(\omega) Z) \\
&=\int a\, \ev(\id\tens\ev\tens\id)(\nabla^L(\omega)\tens Y\tens Z)-\int a\,\mathrm{div}( Y(\omega) Z) \ .
\end{align*}
By nondegeneracy, we deduce that $\mathrm{div}( Y(\omega) Z) - \ev(\id\tens\ev\tens\id)(\nabla^L(\omega)\tens Y\tens Z) =0$, or
\begin{align*}
0 &= \ev(\extd(Y(\omega))\tens Z)+Y(\omega) \, \mathrm{div}(  Z) -  \ev(\id\tens\ev\tens\id)(\nabla^L(\omega)\tens Y\tens Z) \\
&=\ev\big(\omega\tens (\id\tens\ev)(\nabla_\cX\tens\id+\id\tens\hat\nabla)(Y\tens Z)\big),
\end{align*}
where we have used the dual connection $\nabla_\cX$ to $\nabla^L$. 
\end{proof}

Note that (\ref{auxnew}) is weaker than the original auxiliary braid condition $\sigma_{\cX\cX}(X\tens X)=X\tens X$ and provides a kind of {\em improved auxiliary condition}, as the previous one was unnecessarily restrictive. Corollary~\ref{corXreal} gives this as necessary and sufficient for reality of the velocity field in the sense of Theorem~\ref{thmXstar} at all times. In practice, we can assume that $X,\dot X$ are real in this sense and apply $*$ to both sides of the geodesic velocity equations. Comparing with the original velocity equation then gives the improved auxiliary condition as an additional restriction on the space of velocity fields, and we then solve the two together. 
This method will give all real solutions of the velocity equations. We are also free to further restrict our solutions by adopting a particular ansatz, in which case we can solve assuming the ansatz but if the differential equation on the ansatz is not consistent, then the time evolved solution will leave the region where the ansatz is valid.

\subsection{Quantum convected derivative of the divergence} \label{secR}
We start with a quantum analogue of the rate of change of a function along a path parameterised by time $t$, where the velocity of the path is given by the vector field $X$. 

\begin{definition} \label{convder7}
For a function $a:\R\to A$ and a time dependent left vector field $X$ we define
\[
\frac{Da}{Dt} = \frac{\partial a}{\partial t} + \tfrac12\, \ev(\extd a\tens X) + \tfrac12\, \ev \,\hat\sigma( X \tens\extd a) \ .
\]
\end{definition}

This is more symmetric than if we had applied $X$ only on one side, but there is a more concrete reason why we have to make this definition. Recall from Theorem~\ref{thmXstar} part 7) that the divergence of a real vector field obeys the twisted reality condition $a^*=\varsigma(a)$.

\begin{proposition}
Let $a$ be a time-dependent element of $A$ obeying the twisted reality condition $a^*=\varsigma(a)$ and $X$ a time dependent real left vector field. If 
$\ev \,\hat\sigma( \varsigma\circ X \circ \varsigma^{-1} \tens \varsigma\,\omega) =\varsigma\,\ev(X\tens\omega)$ 
 then $\frac{Da}{Dt}$ obeys the twisted reality condition.
\end{proposition}
\begin{proof} Using the definition of $X^*$ in Theorem~\ref{thmXstar}, we get
\[
\ev(\extd a\tens X)^*=\varsigma\,\ev(X^*\tens\extd a^*)=\ev \,\hat\sigma( \varsigma\circ X \circ \varsigma^{-1} \tens \varsigma\,\extd a) =
\varsigma\, \ev \,\hat\sigma( X \tens\extd a)
\]
and
\[
 \ev \,\hat\sigma( X \tens\extd a)^* = \ev( \extd a^*\tens X^*) =  \ev( \varsigma\,\extd a   \tens  \varsigma\circ X \circ \varsigma^{-1} ) = 
\varsigma\, \ev(\extd a\tens X)
\]
as required. \end{proof}

\medskip

\begin{proposition} Let $\hat\nabla$ be obtained from a $*$-preserving bimodule connection $\nabla$ on $\Omega^1$ as in (\ref{hatnabla}). If $X$ is a time-dependent left vector field $X$ obeying the geodesic velocity equation then
\begin{align*}
\frac{D\,\mathrm{div} X }{Dt} &= \ev(\doublenabla(\tilde\ev)(\nabla_\cX X)\tens X) -   (\tilde\ev\tens\ev)  \big(  
(\id\tens ( \id\tens\id - \sigma ))( \nabla_\cX \tens\id
+\id\tens\nabla^L\big)\nabla_\cX X    \tens X \big)  \cr
& \qquad - \tilde\ev (\id\tens\ev\tens\id) ( \nabla_\cX X\tens\nabla_\cX X).  
\end{align*}
\end{proposition}
\begin{proof} We set $\kappa=\tfrac12  \mathrm{div}_{\hat\nabla} X $ in the main part of the geodesic velocity equation,  
\begin{align}
\dot X =[\kappa,X] - (\id\tens X)(\nabla_\cX  X), 
\end{align}
and calculating its divergence using 
 $ \mathrm{div}_{\hat\nabla} Y=\ev\sigma^{-1}_\cX\nabla_\cX Y$ gives
\begin{align} \label{squab}
\mathrm{div}_{\hat\nabla}\dot X = \ev \, \sigma^{-1}_\cX\big( \sigma_\cX(\extd \kappa\tens X)-X\tens\extd \kappa 
-\nabla_\cX((\id\tens X)\nabla_\cX X)\big).
\end{align}
The convected derivative from Definition~\ref{convder7} gives
\begin{equation*}
\frac{D\,\mathrm{div}_{\hat\nabla} X }{Dt} := \mathrm{div}\dot X  + \ev(\extd \kappa\tens X)+\tilde\ev(X\tens\extd \kappa)
= 2\ev(\extd \kappa\tens X)
-  \ev \, \sigma^{-1}_\cX\big( \nabla_\cX((\id\tens X)\nabla_\cX X)\big),
\end{equation*}
where we set $\tilde\ev=\ev\circ \sigma^{-1}_\cX$. Then
\begin{align*}
\nabla_\cX((\id\tens X)\nabla_\cX X) &= \nabla_\cX((\id\tens\ev)(\nabla_\cX X\tens X)) \cr
&=(\id\tens\id\tens\ev) (\nabla_\cX\tens\id\tens\id+\id\tens \nabla^L\tens\id) (\nabla_\cX X\tens X) +(\id\tens\ev\tens\id) ( \nabla_\cX X\tens\nabla_\cX X)  
\end{align*}
\[
2\ev(\extd \kappa\tens X) = \ev(\doublenabla(\tilde\ev)(\nabla_\cX X)\tens X) +  \ev(     (\tilde\ev\tens\id)  (\id\tens \sigma )  \big(  \nabla_\cX \tens\id
+\id\tens\nabla^L\big)(\nabla_\cX X)      \tens X), 
\]
where
\begin{equation}\label{dnablaev}
\doublenabla(\tilde\ev)=\extd\,\tilde\ev-(\tilde\ev\tens\id)  (\id\tens \sigma )  \big(  \nabla_\cX \tens\id
+\id\tens\nabla^L\big),\end{equation}
giving the result. \end{proof}

The last term of this formula for the convective derivative of ${\rm div}(X)$ corresponds to the `kinetic energy' or trace of $(\nabla_\cX X)^2$ term in the classical formula Proposition~\ref{riccinots}. We now give a noncommutative version of this  for a vector field adapted to behave well with respect to $*$. 

\begin{proposition}
If $\nabla$ is $*$-preserving and the conditions of Theorem~\ref{thmXstar} hold and $\nabla\circ\varsigma=(\varsigma\tens \varsigma)\circ\nabla$ then  the `kinetic energy' function
\begin{align*} F(X)={1\over 2}\left(
 \tilde\ev (\id\tens\ev\tens\id) ( \nabla_\cX X\tens\nabla_\cX X)  +  \ev (\id\tens\tilde\ev\tens\id) (\hat\nabla X\tens \hat\nabla X) \right) 
\end{align*}
sends real vector fields to twisted-hermitian elements of the algebra.
\end{proposition}
\begin{proof} We set $\nabla_\cX X=Y_i\tens\xi_i$ for $i=1,2$ as two independent expressions for it. Now
\begin{align*}
 \big(\tilde\ev (\id\tens\ev\tens\id) &( \nabla_\cX X\tens\nabla_\cX X) \big)^*=\big(\tilde\ev(Y_1\tens \ev(\xi_1\tens Y_2) \xi_2)\big)^*
 = \ev( \xi_2{}^* \,\ev(\xi_1\tens Y_2)^*  \tens Y_1{}^*) \\
 &= \ev(\id\tens\tilde\ev\tens\id)( \xi_2{}^* \tens Y_2{}^* \tens \xi_1{}^*  \tens Y_1{}^*) = \ev(\id\tens\tilde\ev\tens\id)(\hat\nabla X^*\tens \hat\nabla X^*)\ .
\end{align*}
If $X$ is real then the result follows from this, as $\varsigma\tens \varsigma$ commutes with $\sigma$ and $\ev(\varsigma\tens \varsigma)=\varsigma\,\ev$ and similarly for $\tilde\ev$ gives the answer. \end{proof}

\medskip 
As a result, we can write
\begin{equation}\label{divRF}
\frac{D\,\mathrm{div} X }{Dt} = -R(X)  - F(X)\end{equation}
where 
\begin{align*} 
R(X)&=-\ev(\doublenabla(\tilde\ev)(\nabla_\cX X)\tens X) +   (\tilde\ev\tens\ev)  \big(  
(\id\tens ( \id\tens\id - \sigma ))( \nabla_\cX \tens\id
+\id\tens\nabla^L)\nabla_\cX X    \tens X \big)  \cr
& \qquad + \tfrac12 \tilde\ev (\id\tens\ev\tens\id) ( \nabla_\cX X\tens\nabla_\cX X)  -\tfrac12 \ev (\id\tens\tilde\ev\tens\id) (\hat\nabla X\tens \hat\nabla X) \end{align*}
plays the role classically of quadratic form $X^k\,X^r \,R_{kr}$ on  $X$ featuring in  Proposition~\ref{riccinots} for the convective derivative of ${\rm div}(X)$. By construction, {\em if} $X_t$ is initially real and obeys the secondary condition in Corollary~\ref{corXreal} to stay real as it evolves then $R(X)$ will likewise be a twisted-hermitian element of the algebra, because the other parts of (\ref{divRF}) are. We therefore propose it as a `quadratic form'  version of the Ricci tensor derived from looking at geodesic velocity vector fields but applicable on any $X\in \cX $. 

\begin{corollary} If $X_t$ obeys the geodesic velocity equation then 
\[ \int \big(R(X)+F(X)-{\rm div}(X)^2\big)=0.\]
\end{corollary}
\begin{proof} We use (\ref{divRF}) in
\[ \int {D{\rm div}(X)\over D t}= \int {\extd\, {\rm div}(X)\over \extd t}+X(\extd\, {\rm div}(X)={\extd \over\extd t}\int{\rm div}(X)-{\rm div}(X)^2=-\int{\rm div}(X)^2.\]
\end{proof}

\section{Quantum geodesics on $M_2(\C)$}\label{secmat}

We take $M_2(\C)$ with its standard differential calculus $\Omega=M_2[s,t]/\<s^2,t^2\>$, where $s$ and $t$ are central and
$\extd f=(\del_sf)s+(\del_tf)t$ is given on$f\in M_2(\C)$ by
\[ \del_s f= [E_{12},f],\quad \del_t f=[E_{21},f].\]
The calculus is inner  with $\theta=E_{12}s+E_{21}t$ and 
$\extd$ on higher forms is likewise given by a graded commutator, in particular
\[ \extd s=2\theta s=2E_{21}st,\quad \extd t=2\theta t=2E_{12}st.\]
There is a natural  $*$ structure $s^*=-t$ and we take lift  $i(st)= {1\over 2}(s\tens t+ t\tens s)$. Here $i((st)^*)=i(st)=i(st)^\dagger$ so commutes with $*$ rather than anticommuting as required to ensure the usual reality properties of Ricci. Metrics are given by four complex coefficients in the tensor product basis with condition for `reality' $g^\dagger=g$ of the metric, see\cite{BegMa}. Here we consider just the metric
\[ g=s\tens s+t\tens t\]
as a sample in this moduli space. The QLCs for this are not unique but there is a  natural 3-parameter moduli of inner connections defined solely by the generalised braiding as a bimodule map\cite[p. 776]{BegMa}
 \begin{align} \label{prid}
 \sigma_L=\begin{pmatrix}  1-\mu & \rho & -\rho & -\nu \\
 \rho+{\mu(\mu+\nu-2)\over \rho} & -\mu & \mu-1 & -\rho \\
 \rho & \nu-1 & -\nu & -\rho-{\nu(\mu+\nu-2)\over\rho}\\
 -\mu & \rho & -\rho & 1-\nu \end{pmatrix};\quad \mu,\nu,\rho\in \C, 
\end{align}
where the conventions are such that the second row gives the coefficients of $\sigma_L(s\tens t)$
in basis order $s\tens s$, $s\tens t $,   $t\tens s$,   $t\tens t$. This corresponds to the left connection on $\Omega^1$ given by
\begin{align*}\nabla^L s= - \Gamma^s{}_{bc} \, b\tens c   = 2E_{21} t\tens s&+ \Big(\mu E_{12}-(\rho +(\mu+\nu-2){\mu\over\rho})E_{21}\Big)s\tens s\\
&+(\mu E_{21}-\rho E_{12})(s\tens t-t\tens s)+(\nu E_{12}+\rho E_{21})t\tens t,\\
\nabla^L t   = - \Gamma^t{}_{bc} \, b\tens c = 2E_{12}s\tens t&+(\mu E_{21}-\rho E_{12})s\tens s-(\nu E_{12}+\rho E_{21})(s\tens t-t\tens s)\\
&+\Big(\nu E_{21}+(\rho +(\mu+\nu-2){\nu\over\rho})E_{12}\Big)t\tens t\end{align*}
which defines the Christoffel symbols $\Gamma^a{}_{bc}$ for $a,b,c\in\{s,t\}$. The connection is $*$-preserving which $\sigma\dagger\sigma=\dagger$, which is when $\bar\rho=-\rho,\nu=\bar\mu$. 

The dual right connection on $\cX$ is then given by 
$\nabla_\cX f_e=\Gamma^a{}_{be} \, f_a\tens b$. For the basis order $s\tens f_s$, $s\tens f_t$, $t\tens f_s$, $t\tens f_t$ on $\Omega^1\tens \cX$ and $f_s\tens s$, $f_s\tens t$, $f_t\tens s$, $f_t\tens t$ on $\cX\tens \Omega^1$
and using the same conventions as (\ref{prid}), its braiding is therefore
 \begin{align} 
\label{pridvec}
 \sigma_\cX=\begin{pmatrix} 1-\mu & -\rho & \rho & -\nu  \\
 \rho & -\nu & \nu-1 &   -\rho-{\nu(\mu+\nu-2)\over\rho}   \\
   \rho+{\mu(\mu+\nu-2)\over \rho}  &  \mu-1 & -\mu & -\rho   \\
    -\mu & -\rho & \rho & 1-\nu     \end{pmatrix};\quad \mu,\nu,\rho\in \C. 
\end{align}

\begin{proposition}\label{propM2div} The geometric divergence for the above inner connections agrees with the state divergence for a nondegenerate positive linear functional $\phi$ if and only if $\mu=\nu=0$ and $\phi={1\over 2}{\rm Tr}$.  Then ${\rm div}(f_s X^s+f_t X^t)=[E_{12},X^s]+ [E_{21},X^t]$ on a vector field in $\cX$ with components in a dual basis $f_s,f_t$. Moreover, $f_s^*=-f_t$. 
\end{proposition}
\begin{proof} We first calculate $\hat\nabla$ the details of which are omitted. This then gives
\begin{align*}
\mathrm{div}_{\hat\nabla}(f_s) &= E_{12}(\mu+\nu)+E_{21}(2-\mu-\nu)\frac{\mu}{\rho} \cr
\mathrm{div}_{\hat\nabla}(f_t) &= E_{21}(\mu+\nu)-E_{12}(2-\mu-\nu)\frac{\nu}{\rho}.  
\end{align*}
where $f_s,f_t$ are the dual basis vector fields in $\cX $. For comparison, if we set $\phi(a)=\mathrm{Tr}(aN)$ for $N$ a trace 1 positive matrix which is invertible for $\phi$ to be non-degenerate. Then the state divergence is
\begin{align}\label{divphiM2}
\mathrm{div}_\phi(f_s) &= E_{12}-NE_{12}N^{-1},\quad \mathrm{div}_\phi(f_t) = E_{21}-NE_{21}N^{-1}  
\end{align}
and the only case where these coincide is $N=\tfrac12 I_2$ and $\mu=\nu=0$. Hence we restrict now this one-parameter moduli space given by $\rho$ where the geometric divergence and the $\phi$-divergence coincide, where $\phi={1\over 2}{\rm Tr}$. For the $*$-structure on $\cX $ in Theorem~\ref{thmXstar}, we proceed with $\mu=\nu=0$ and
using
 \[ \hat\sigma(f_s\tens s)=s\tens f_s+\rho\,s\tens f_t-\rho\,t\tens f_s\ ,\ 
\hat\sigma(f_s\tens t)=\rho\,s\tens f_s -   t\tens f_s-\rho\,t\tens f_t
\]
we have 
\[ f_s^*(s)=(\ev\hat\sigma(f_s\tens (-t)))^*=0,\quad f_s^*(t)=(\ev\hat\sigma(f_s\tens (-s)))^*=-1\]
and similarly for $f_t$. 
\end{proof}

In view of this lemma, we now take $\int={1\over 2}{\rm Tr}$ as our preferred state. Note that there is also a significant 4-parameter moduli of QLCs with $\sigma=-{\rm flip}$ on the basis elements and admitting a bimodule map $\alpha:\Omega^1\to \Omega^1\tens_A\Omega^1$ going beyond the inner case\cite[Exercise~8.3]{BegMa}. Our analysis above can be extended to this wider class as well as repeated for other metrics. 

\subsection{Ricci quadratic form for general $\rho$}

Proceeding with the $\mu=\nu=0$, we still have a 1-parameter moduli of QLCs,
\begin{align*}\nabla^L s=2E_{21} t\tens s&-\rho E_{12}(s\tens t-t\tens s)+\rho E_{21}(t\tens t-s\tens s),\\
\nabla^L t=2E_{12}s\tens t&-\rho E_{21}(s\tens t-t\tens s)+\rho E_{12}(t\tens t-s\tens s).\end{align*}
where $\rho$ is an imaginary parameter. The curvature is 
\begin{align*}
R_{\nabla^L}s&=2(E_{11}-E_{22}) s\wedge t \tens(\rho t+s) +2\rho^2 s\wedge t\tens s  + 4 E_{22} s\wedge t\tens s
-2 (E_{11}-E_{22}) s\wedge t \tens \rho t \cr
&=    2(1+\rho^2 )s\wedge t\tens s  
\cr
R_{\nabla^L} t&= 2(E_{11}-E_{22}) s\wedge t \tens (\rho s-t) +2\rho^2 s\wedge t\tens  t
+ 4 E_{11} s\wedge t\tens  t-2(E_{11}-E_{22}) s\wedge t \tens \rho s \cr
&= 2(1+\rho^2 )s\wedge t\tens t
\end{align*}
The usual Ricci tensor defined by the canonical lift $i$ is\cite[Eqn. (8.21)]{BegMa}
\begin{equation}\label{riccimat} {\rm Ricci}=(1+\rho^2)(s\tens t+t\tens s),\end{equation}
which is still hermitian with the $*$-structures (even though $i$ is not suitably antihermitian), but not quantum symmetric. Here  $\rho=\pm\imath$ gives a pair of natural flat QLCs on $M_2(\C)$. We aim to contrast this with what we get for the Ricci quadratic form.

From the above, the dual right  connection on vector fields now simplifies to 
\begin{align*}\nabla_\cX f_s&= \rho\,E_{21}\, f_s\tens s -(\rho\,E_{12}+2\,E_{21})f_s\tens t+\rho\,E_{12} f_t\tens s-\rho\,E_{21} f_t\tens t, 
 \\
\nabla_\cX f_t&= \rho\,E_{12}    f_s\tens s    - \rho\,E_{21}     f_s\tens t     +(   \rho\,E_{21}  - 2\,E_{12})     f_t\tens s  - \rho\, E_{12}      f_t\tens t, \\
\sigma_\cX(s\tens f_s)&= f_s\tens s+\rho(f_t\tens s-f_s\tens t),\qquad \sigma_\cX(s\tens f_t)=-f_t\tens s+\rho(f_s\tens t-f_t\tens t),  \\
\sigma_\cX(t\tens f_s)&= -f_s\tens t+\rho(f_s\tens s-f_t\tens t),\qquad \sigma_\cX(t\tens f_t)=f_t\tens t+\rho(f_t\tens s-f_s\tens t).
\end{align*}

In what follows, we adopt an index notation with $a=s,t$ and 
\[ D_bX^a=\del_bX^a+\Gamma^a{}_{bc}X^c\]
so that $\nabla_{\cX} X=f_a\tens b D_bX^a$ if  $X= f_s X^s+ f_t X^t\in \cX $, where we sum over repeated indices. 

\begin{proposition}   $\doublenabla(\tilde\ev)=0$. Moreover, writing  $Y_{ab}=D_bX^a$ for brevity,
\[ F(X)=Y_{ab}Y_{ba} + \rho\, (Y_{ts}Y_{ss} + Y_{ss} Y_{ts}- Y_{tt}Y_{st}-Y_{st}Y_{tt})+ \rho^2\,( (Y_{st}+Y_{ts})^2-(Y_{ss}+Y_{tt})^2)\]
\begin{align*} 
R(X)&=-2(X^sX^t+X^tX^s) +2\rho\Big(  (X^s){}^2-(X^t){}^2  + \left(
\begin{array}{cc}
 b_s b_t-c_s c_t & -a_s
   c_t-c_s d_t \\
 a_s b_t+b_s d_t & b_s b_t-c_s
   c_t \\
\end{array}
\right)
 \Big) \\
 &\quad + \rho^2\left(
\begin{array}{cc}
 -b_s^2+b_t^2+c_s^2-c_t^2 & (a_s + d_s)
   (b_t+c_s)- (a_t+d_t) (b_s+c_t)
   \\
 -(a_s + d_s) (b_s+c_t)+(a_t+ d_t)
   (b_t+c_s)&
   -b_s^2+b_t^2+c_s^2-c_t^2 
\end{array}
\right)  
 \end{align*}
where
\[ X^s=\left(\begin{array}{cc} a_s  &  b_s  \\  c_s &  d_s  \end{array}\right)
\ ,\ 
X^t=\left(\begin{array}{cc} a_t  & b_t  \\ c_t  &d_t \end{array}\right)\ .\]
\end{proposition}
\begin{proof} We first calculate, with $\hat\nabla=\sigma^{-1}_\cX\nabla_\cX$, 
\begin{align*} 
 \tilde\ev (\id\tens\ev\tens\id) ( \nabla_\cX X\tens\nabla_\cX X) &= Y_{ab}Y_{ba} \\
  \ev (\id\tens\tilde\ev\tens\id) (\hat\nabla X\tens \hat\nabla X) &= Y_{ab}Y_{ba} +2 \rho\, (Y_{ts}Y_{ss} + Y_{ss} Y_{ts}- Y_{tt}Y_{st}-Y_{st}Y_{tt})\\
  &\quad +2 \rho^2\,( (Y_{st}+Y_{ts})^2-(Y_{ss}+Y_{tt})^2)
 \end{align*}
 giving $F(X)$ as stated. Next, we calculate 
\[
\doublenabla(\tilde\ev)=\extd\,\tilde\ev-(\tilde\ev\tens\id)  (\id\tens \sigma )  \big(  \nabla_\cX \tens\id
+\id\tens\nabla^L\big)
\]
and as it is a right module map we only need to evaluate it on $f_a\tens b$ for $a,b\in\{s,t\}$. First $\extd\,\tilde\ev(f_a\tens b)=0$ and then, summing over repeats
\[
\nabla_\cX f_a\tens b+f_a\tens\nabla^Lb=\Gamma^c{}_{ea} f_c\tens e\tens b - \Gamma^b{}_{en} f_a\tens e \tens n
\]
Now using the formula for $\sigma_L$, for example
\[
 (\tilde\ev\tens\id)  (\id\tens \sigma )  \big(  \nabla_\cX \tens\id
+\id\tens\nabla^L\big)(f_s\tens s)= \rho(\Gamma^t{}_{ss}-\Gamma^s{}_{st})s+\rho(\Gamma^s{}_{tt}-\Gamma^t{}_{ts})t
\]
and substituting for the Christoffel symbols and looking at the other cases
gives $\doublenabla(\tilde\ev)=0$ as claimed. Next,  we calculate
\[
(\tilde\ev\tens\id)   
(\id\tens ( \id\tens\id - \sigma ))     ( \nabla_\cX \tens\id +\id\tens\nabla^L)\nabla_\cX X    
\]
in several parts. We start with
\begin{align*}
 (\tilde\ev\tens\id)  \big(  \nabla_\cX \tens\id
+\id\tens\nabla^L\big)(f_s\tens s) &= -\rho\, E_{21}s-\rho\,E_{12}t , \\
 (\tilde\ev\tens\id)  \big(  \nabla_\cX \tens\id
+\id\tens\nabla^L\big)(f_s\tens t) &= -\rho\, E_{12}s+(2\,E_{12} -\rho\,E_{21}     )t , \\
 (\tilde\ev\tens\id)  \big(  \nabla_\cX \tens\id
+\id\tens\nabla^L\big)(f_t\tens s) &= (2\,E_{21}+\rho\, E_{12}) s  +   \rho\,E_{21}t , \\
 (\tilde\ev\tens\id)  \big(  \nabla_\cX \tens\id
+\id\tens\nabla^L\big)(f_t\tens t) &= \rho\, E_{21}s + \rho\,E_{12}t  \ .
\end{align*}
\begin{align*}
(\tilde\ev\tens\id)   
(\id\tens ( \id\tens\id - \sigma ))  &   ( \nabla_\cX \tens\id +\id\tens\nabla^L)(f_a\tens b\,Y_{ab}   )    \\
&= (\tilde\ev\tens\id)   
(\id\tens ( \id\tens\id - \sigma ))     ( \nabla_\cX \tens\id +\id\tens\nabla^L)(f_a\tens b)  .Y_{ab}  \\
&\quad + (\tilde\ev\tens\id)   
(\id\tens ( \id\tens\id - \sigma ))   (f_a\tens \sigma_L(b\tens\extd Y_{ab} ))   \\
&= (\tilde\ev\tens\id)      ( \nabla_\cX \tens\id +\id\tens\nabla^L)(f_a\tens b)  .Y_{ab}  \\
&\quad + (\tilde\ev\tens\id)   
(\id\tens (  \sigma_L - \id\tens\id  ))   (f_a\tens b\tens c) .[E_c,Y_{ab} ]
\end{align*}
and evaluating this against $X$ gives
\begin{align*}
&= (\tilde\ev\tens\ev)   (   ( \nabla_\cX \tens\id +\id\tens\nabla^L)(f_a\tens b) \tens f_c) .Y_{ab}  X_c \\
&\quad + (\tilde\ev\tens\ev)   
(\id\tens (  \sigma_L - \id\tens\id  ) \tens\id)   (f_a\tens b\tens c\tens f_d) .[E_c,Y_{ab} ]\, X_d.
\end{align*}
Now, $R(X)$ consists of this plus $-\rho\,(Y_{ts}Y_{ss} + Y_{ss} Y_{ts}- Y_{tt}Y_{st}-Y_{st}Y_{tt}) - \rho^2\, ( (Y_{st}+Y_{ts})^2-(Y_{ss}+Y_{tt})^2)$ as already computed for $F(X)$, giving $R(X)$ as stated.  Explicitly,
\[Y_{ss}=[E_{12},X^s]+\rho (E_{21}X^s+ E_{12}X^t),\quad Y_{st}=-\{E_{21},X^s\}-\rho(E_{12}X^s+ E_{21} X^t),\]
\[ Y_{tt}=[E_{21},X^t]-\rho ( E_{21}X^s+E_{12}X^t),\quad Y_{ts}=-\{E_{12},X^t\}-\rho( E_{12} X^s+E_{21}X^t),\]
which one can further compute in terms of the entries of $X^s,X^t$. \end{proof}

If $\rho=0$ then we see that $R(X)$ agrees with contraction against the Ricci curvature tensor (\ref{riccimat}),  bearing in mind that the latter is $-{1/2}$ of the classically normalised one. But for $\rho\ne 0$, we see that the two approaches are a little different in this example. 

%Explicitly in terms of the matrix entries of $X^s, X^t$, 
%\[Y_{ss}=\left(\begin{array}{cc}c_s+c_t \rho  & -a_s+d_s+\rho d_t   \\\rho a_s   &\rho b_s  -c_s \\
%\end{array}\right),\quad Y_{st}=\left(\begin{array}{cc} -b_s-\rho c_s   & - \rho d_s  \\ -a_s-\rho a_t -d_s & - b_s-\rho b_t   \\\end{array}\right)\]
%\[Y_{ts}=\left(\begin{array}{cc}\rho c_s  -c_t & -a_t+\rho d_s  -d_t \\ \rho a_t   & \rho b_t  -c_t \\\end{array}\right),
%\quad Y_{tt}=\left(\begin{array}{cc} -b_t-\rho c_t   & -\rho d_t   \\ -\rho a_s  +a_t-d_t & b_t-\rho b_s   \\\end{array}\right)\]

\subsection{Quantum geodesic flow equations on $M_2(\C)$}

As before, we write $X=f_s  X^s+ f_t X^t$ but with components now time-dependent, and we set
\[ \kappa={1\over 2}{\rm div}(X)={1\over 2}  \big([E_{12},X^s]+[E_{21},X^t] \big)
\]
along with preferred state $\int={1\over 2}{\rm Tr}$ in view of Proposition~\ref{propM2div}.  The reality condition on vector fields from Theorem~\ref{thmXstar} comes down to 
\[ (X^s)^*=-X^t\]
since $\varsigma=\id$, while the  geodesic velocity equation is
\begin{align*} 
\dot X &=\frac12  \big[ ([E_{12},X^s]+ [E_{21},X^t]),X \big]- (\id\tens X )\nabla X \\
&= \frac12  \big[ ([E_{12},X^s]+ [E_{21},X^t]),X \big] - f_s X(\extd X^s) - f_t X(\extd X^t) - 
(\id\tens X )((\nabla_\cX f_s)X^s)- (\id\tens X )((\nabla_\cX f_t)X^t).
\end{align*}
We have
\begin{align*}
&(\id\tens X )((\nabla_\cX f_s)X^s)+ (\id\tens X )((\nabla_\cX f_t)X^t) \cr
&\quad =  f^s \big(  E_{21} X^sX^s \rho-(2 E_{21} +\rho E_{12})X^sX^t + E_{12} X^t X^s \rho-E_{21} X^t X^t \rho \big) \cr
&\qquad +  f^t  \big(     E_{12} X^sX^s\rho- E_{21} X^s X^t \rho -(2E_{12}-E_{21}\rho) X^t X^s-E_{12} X^tX^t \rho
\big)
\end{align*}
and thus 
\begin{align} \nonumber
\dot X^s
&= \frac12  \big[ ([E_{12},X^s]+ [E_{21},X^t]),X^s \big] -  [E_{12},X^s] X^s  - [E_{21},X^s] X^t  \\  
\label{velmata} &\qquad  - \big(  E_{21} X^sX^s \rho-(2 E_{21} +\rho E_{12})X^sX^t + E_{12} X^t X^s \rho-E_{21} X^t X^t \rho \big), \\
\nonumber  \dot X^t
&= \frac12  \big[ ([E_{12},X^s]+ [E_{21},X^t]),X^t \big] -  [E_{12},X^t] X^s  - [E_{21},X^t] X^t \\ & 
 \label{velmatb}  \qquad - 
 \big(     E_{12} X^sX^s\rho- E_{21} X^s X^t \rho -(2E_{12}-E_{21}\rho) X^t X^s-E_{12} X^tX^t \rho
\big) \ .
\end{align}

We first look at the content of the original auxiliary braid condition in \cite{BegMa:geo} which is sufficient but not necessary for real geodesic evolution. 

\begin{lemma} If $X\in \cX $ is real then $\sigma_{\cX,\cX}(X\tens X)=X\tens X$ holds only if $X=0$.
\end{lemma}
\begin{proof}  Looking at the coefficients of $\sigma_{\cX,\cX}(X\tens X)-X\tens X=0$ we get one entry $\rho(X^s X^t+X^t X^s)=0$ so if $\rho\neq 0$ then $X^s X^t+X^t X^s=0$. There is another entry $-(X^s X^t+X^t X^s)+\rho(X^s X^s+X^t X^t)=0$, so if $\rho=0$ we still get  $X^s X^t+X^t X^s=0$. In the case of real $X$ this means $X^s{}^*X^s+X^sX^s{}^*=0$ which requires $X^s=X^t=0$. \end{proof}

Thus, this condition is too strong. Instead, we proceed with $X$ real but the minimal `improved auxiliary condition' on it such that the flow remains real. This is obtained by applying $*$ to the equation for $\dot X^s$ and requiring to get the equation for $\dot X^t$, and comes down to
\begin{equation}\label{auxnewmat}\{E_{12},\{X^s,X^t\}-\rho((X^s){}^2-(X^t){}^2)\}+\rho \left([E_{21},[X^s,X^t]] \}\right)=0\end{equation}
on remembering that $\rho$ is imaginary.  If $\rho=0$ then we need $E_{12}$ to anticommute with $\{X^s,X^s{}^*\}$, which is a positive operator and hence must be real and symmetric. One can show that this requires $\{X^s,X^s{}^*\}=0$, so we are back with the original auxiliary braid condition case. Therefore, real solutions in the sense $X^s{}^*=-X^t$ at all times $t$ can exist only with $\rho\ne 0$.

After solving for $X$, we then solve the amplitude flow equation $\nabla_E\psi=0$ for $\psi(t)\in M_2(\C)$, which amounts to
\begin{equation}\label{psimat} \dot\psi=-[E_{12},\psi]X^s-[E_{21},\psi]X^t-\psi\kappa.\end{equation}

\subsection{Quantum geodesic flows for $\rho=\imath$} The choice $\rho=\imath$, although flat, seems to be a natural choice after we excluded $\rho=0$ and we solve this here. We first solve the improved auxiliary equation (\ref{auxnewmat}) with a 4-parameter solution 
\[ X^s=\begin{pmatrix}e^{\pi\imath\over 4}(d-a) & c+\imath b\\ b+\imath c & e^{\pi\imath\over 4}(d+a)\end{pmatrix};\quad  a,b,c,d\in \R. \]
We now let these parameters vary in time and find that (\ref{velmata}) is 
\[ \dot c=0,\quad \dot d=0,\quad \dot a=2 b d,\quad \dot b=-2 a d,\]
which is simple harmonic motion. Hence, we have a 4-parameter geodesic velocity field
\[ X^s=\begin{pmatrix}e^{\pi\imath\over 4}(\delta-(\alpha \cos(2\delta t)+\beta\sin(2 \delta t))& \gamma+\imath (\beta\cos(2\delta t)-\alpha\sin(2\delta t))\\ (\beta\cos(2\delta t)-\alpha\sin(2\delta t))+\imath \gamma & e^{\pi\imath\over 4}(\delta+(\alpha \cos(2\delta t)+\beta\sin(2 \delta t)))\end{pmatrix},\quad X^t=-X^s{}^*\]
for initial values $\alpha,\beta,\gamma,\delta$ of $a,b,c,d$ respectively.

\begin{figure}
\[ \includegraphics[scale=1.1]{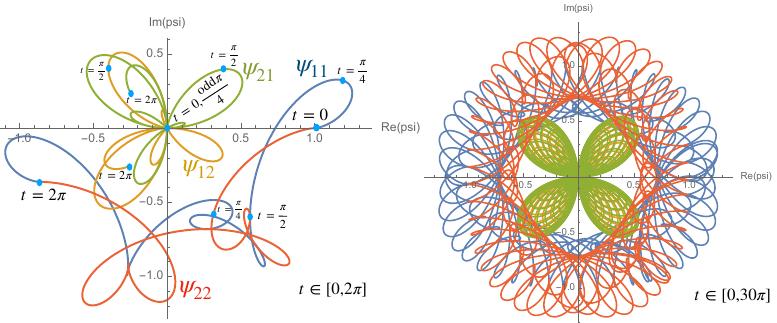}\]
\caption{Quantum geodesic on $M_2(\C)$ with its flat connection $\rho=\imath$ and a geodesic vector field given by $\alpha=\gamma=\delta=1, \beta=0$. We plot the matrix entries on the complex plane of the resulting amplitude flow $\psi(t)\in M_2(\C)$ with initial  $\psi(0)=\id$. The outer curves are the diagonal entries and the inner ones off-diagonal. \label{matgeo}}
\end{figure} 

Relative to this, we have to solve the amplitude flow equation (\ref{psimat}) where
\[ \kappa=  \begin{pmatrix}\beta\cos(2\delta t)-\alpha\sin(2\delta t)&e^{\pi\imath\over 4}(\alpha \cos(2\delta t)+\beta\sin(2 \delta t))\\
e^{-{\pi\imath\over 4}}(\alpha \cos(2\delta t)+\beta\sin(2 \delta t)) & -(\beta\cos(2\delta t)-\alpha\sin(2\delta t))\end{pmatrix}.\]
A numerical solution for $\alpha=\gamma=\delta=1, \beta=0$ and $\psi(0)=\id$ the identity matrix is shown in Figure~\ref{matgeo}. One can check to within numerical accuracy that the off-diagonal entries of $\psi$ vanish at $t=n\pi/4$ for all odd $n$ in the range and at these points ${\rm Im}(\psi_{11})={\rm Re}(\psi_{22})$. Meanwhile at $t=n\pi/2$ for all $n$ in range, one has $\psi_{11}=\psi_{22}$, $\psi_{12}=(-1)^n\bar\psi_{21}$ and $\psi_{12}=r_n e^{(-1)^n{\pi\imath\over 4}}$ for real $r_n$.  We have marked $t=\pi/4$ and $t=\pi/2$ as examples. Looking from a coarser perspective out to large $t$, we also see that the diagonal entries of $\psi$ precess in an approximate circle while the off-diagonal entries remain near to 0 and repeatedly return to it. This reflects the initial starting point. One can further check that while the squared absolute values of the 4 entries of $\psi$ clearly vary in time (as the square of the distance from the origin in the complex plane), their sum remains constant at 2, in line with our probabilistic interpretation with respect to $\int={1\over 2}{\rm Tr}$.

\section{Quantum geodesics on the fuzzy sphere}\label{secfuz}

The fuzzy sphere or coadjoint quantisation as an orbit in $su_2^*$ is the algebra
\[ [x_i,x_j]=2\imath\lambda_p \epsilon_{ijk}x_k,\quad \sum_i x_i^2= 1-\lambda_p^2,\]
where $0\le \lambda_p<1$ is a real deformation parameter. We use the  3D calculus recently introduced in \cite[Example~1.46]{BegMa} 
with central basis $s^i$, $i=1,2,3$ with 
\[  \extd x_i= \epsilon_{ijk}x_{j}s^k.\]
The calculus is inner (in degree 1 only) with 
\[ \theta= \tfrac{1}{2\imath \lambda_p}x_i s^i=\tfrac{1}{(2\imath\lambda_p)^2}x_i\extd x_i=-\tfrac{1}{(2\imath\lambda_p)^2}(\extd x_i) x_i\]
so that the partial derivatives are the `orbital angular momentum' derivations
\[ \del_i f= \tfrac{1}{2\imath \lambda_p}[x_i, f],\quad \del_i x_j=\eps_{ijk}x_k,\quad \eps_{ijk} \del_i\del_j=\del_k. \]
The natural exterior algebra is
\[  s^i \wedge s^j + s^j \wedge s^i =0,\quad \extd s^i =-\tfrac{1}{2}\epsilon_{ijk}s^j \wedge s^k\]
but note that this is no longer inner in higher degree by $\theta$. The $*$-structure is $x_i,s^i$ self adjoint and the canonical lift from 2-forms is $i(s^i\wedge s^j)={1\over 2}(s^i\tens s^j-s^j\tens s^i)$ as classically. 

A general quantum symmetric metric has the form $g=g_{ij}s^i\tens s^j$ for $g_{ij}$ a positive real symmetric matrix and a left connection takes the form (dropping the $1\over 2$ factor compared to the definition of $\Gamma$ in \cite{LirMa}), 
\[ \nabla^L s^i= -  \Gamma^i{}_{ jk } s^j\tens s^k\]
for some coefficients $\Gamma^i{}_{jk}$ which we assume for simplicity to be constants (multiples of the identity $1\in A$). We have a bimodule connection with $\sigma_L$ just the flip on the basis, and for the connection to be $*$-preserving we then need  $\Gamma^i{}_{jk}$ to be real. We do not limit ourselves to a QLC this was found in \cite{LirMa} to be unique among such connections, namely
\begin{equation}\label{s2qlc} \nabla^L s^i= - g^{il} (\eps_{lkm}g_{mj}+{{\rm Tr}(g)\over 2}\eps_{ljk})s^j\tens s^k. 
\end{equation}
 The torsion for a general connection in our class is
\[ T_\nabla(s^i)=-{1\over 2}T^i{}_{jk}s^j\wedge s^k;\quad T^i{}_{jk}=\Gamma^i{}_{jk}-\Gamma^i{}_{kj}-\eps_{ijk}\]
which can be shown to vanishes for the QLC. The Riemann curvature for a general connection has the form
\[ R_{\nabla^L}(s^i)=\rho^i{}_{jk}\eps_{jmn}s^m\wedge s^n\tens s^k,\quad \rho^i{}_{jk}=\tfrac{1}{2}(\Gamma^{i}{}_{jk}-\eps_{jmn}\del_m\Gamma^{i}{}_{nk}-\eps_{jmn}\Gamma^{i}{}_{ml}\Gamma^{l}{}_{nk}).\]
and the Ricci curvature the form
\begin{equation}\label{Rmn} R_{mn}=\rho^i{}_{jn}\eps_{jim}={1\over 2}\left(\Gamma^i{}_{jn}\eps_{imj}+\del_m\Gamma^i{}_{in}-\del_i\Gamma^i{}_{mn}+\Gamma^i{}_{mj}\Gamma^j{}_{in}-\Gamma^i{}_{ij}\Gamma^j{}_{mn}\right),\end{equation}
albeit in our constant case we do not need the derivative terms. Its value for the QLC is computed in \cite{LirMa} 
%\[ R_{st}=\tfrac{1}{2}\eps^i{}_{tm}g_{mj}\eps_{jis}+ \tfrac{1}{2}{\rm Tr}(g^{-1})g_{st}-\tfrac{1}{2}\delta_{st}+\tfrac{1}{2}{\rm Tr}(g)(g^{st}-{\rm Tr}(g^{-1})\delta_{st}-\tfrac{1}{2}\eps^{ij}{}_p\eps_{ijs}g_{pt})+\tfrac{1}{8}{\rm Tr}(g)^2\eps^{ij}{}_t\eps_{ijs},\]
which then gives the scalar curvature in this case as
\[ S= \tfrac{1}{2}({\rm Tr}(g^2)-\tfrac{1}{2}{\rm Tr}(g)^2)/\det(g). \]

If we introduce the (central) basis $f_i$ for left vector fields which is dual to the $s^i$ then the right connection dual to the previous left connection is
\[
 \nabla_\cX f_k= \Gamma^i{}_{ jk } f_i\tens s^j
\]
with $\sigma_\cX$ the flip among the $f_i$ and $s^j$. This implies on a left vector field $X=f_k X^k$, 
\[
\nabla_\cX X = \Gamma^i{}_{ jk }  f_i\tens s^jX^k + f_i\tens 
\del_j X^i s^j =D_j X^i f_i\tens s^j,\]
where
\[ D_j X^i=   \del_j X^i + \Gamma^i{}_{ jk }  X^k.   
\]

We also need a natural state and here we use $\int:A\to \C$ defined by an expansion in terms of noncommutative spherical harmonics\cite{ArgMa} (as the $SU_2$-invariant component). This is known to be a trace and was used recently in the construction of the Dirac operator on the fuzzy sphere\cite{LirMa2}. 

\begin{proposition}\label{propfuzdiv} The geometric divergence and the $\int$-divergence agree on the fuzzy sphere if $\Gamma^i{}_{ij}=0$ for all $j$, which includes the case of the QLC. In this case ${\rm div}(f_kX^k)=\del_k X^k$ and $f_k^*=f_k$. 
\end{proposition}
\begin{proof} The inverse of $\sigma_\cX$ is also the flip on the basis and gives the left connection $\hat\nabla$ and hence
$\mathrm{div}_{\hat\nabla}(f_k) =\Gamma^i{}_{jk}\delta_{ij}=\Gamma^i{}_{ik}$, which   vanishes for the QLC read off from (\ref{s2qlc}). Then ${\rm div}_{\hat\nabla}(f_ka)={\rm div}_{\hat\nabla}(af_k)=\ev(\extd a\tens f_k)=\del_ka$ for any vector field $f_ka$. Hence, from Proposition~\ref{propdiveq} for compatibility of the two types of divergence, we need 
\[ \int \del_j a=\int [ \tfrac{1}{2\imath \lambda_p}x_j,a]=0\]
for any $a$, which holds as $\int$ is a trace. Since $s^j{}^*=s^j$, the $*$-operation in Theorem~\ref{thmXstar} is then as stated. \end{proof}

We proceed with connections where $\Gamma$ obeys the condition in Proposition~\ref{propfuzdiv} so that we can use Theorem~\ref{thmXstar}. 

\begin{proposition}\label{propRfuz} On $X=f_iX^i$ for $X^i\in \C_\lambda[S^2]$, the kinetic form and the quadratic Ricci form are 
\begin{align*} F(X)=(D_j X^i) (D_i X^j),\quad 
R(X)&=-2(R_{ij}+\eps_{mni}\Gamma^m{}_{nj})X^jX^i-T^k{}_{ij}(D_k X^i)X^j.
 \end{align*} 
\end{proposition}
\begin{proof} Clearly the two halves of $F(X)$ are each half of 
\[ \tilde\ev(\id\tens\ev\tens\id)(\nabla_\cX X \tens \nabla_\cX X)=(D_j X^i)(D_n X^m) \tilde\ev(\id\tens\ev\tens\id)(f_i\tens s^j\tens f_m\tens s^n)=(D_j X^i)(D_i X^j),\]
\[ \ev(\id\tens\tilde\ev\tens\id)(\hat\nabla X \tens \hat\nabla X)=(D_j X^i)(D_n X^m) \ev(\id\tens\tilde\ev\tens\id)(s^j\tens f_i\tens  s^n\tens f_m)=(D_j X^i)(D_i X^j),
\]
where ${\hat\nabla} X=\sigma^{-1}_\cX\nabla_\cX X=D_j X^i s^j\tens f_i$ as the $s^i,f_j$ are both central so we can pull all coefficients to the left and $\sigma_\cX$ is just the flip on the basis. 

Being equal, these do not contribute to $R(X)$. The first term of $R(X)$ also vanishes as
\begin{align*}\ev(\doublenabla(\tilde\ev)(\nabla_\cX X)\tens X)=\ev(\doublenabla(\tilde\ev)(f_i\tens s^j)D_jX^i\tens f_l)X^l=0
\end{align*}
using that $\doublenabla(\tilde\ev):\cX\tens\Omega^1\to \Omega^1$ is a right module map and that, from  (\ref{dnablaev}), 
\begin{align*} \doublenabla(\tilde\ev)(f_i\tens s^j)&=-(\tilde\ev\tens\id)(\id\tens\sigma)(\Gamma^m{}_{ni}f_m\tens s^n\tens s^j-f_i \tens\Gamma^j{}_{mn}s^m\tens s^n)\\
&=-(\tilde\ev\tens\id)(\Gamma^m{}_{ni}f_m\tens s^j\tens s^n-f_i \tens\Gamma^j{}_{mn}s^n\tens s^m)\\
&=-\Gamma^j{}_{ni} s^n+ \Gamma^j{}_{mi}s^m=0.
\end{align*}

It remains to calculate the middle term,
\begin{align*} (\tilde \ev\tens\ev)&\left((\id\tens(\id\tens\id-\sigma))(\nabla_\cX\tens\id+\id\tens\nabla^L)   (D_jX^i f_i\tens s^j)\tens f_m X^m\right)\\
&=(D_k(D_l X)^i-D_jX^i\Gamma^j{}_{kl})(\tilde \ev\tens\ev)\left((\id\tens(\id\tens\id-\sigma))(f_i\tens s^k\tens s^l)\tens f_m X^m\right)\\
&=(D_k(D_l X)^i-D_jX^i\Gamma^j{}_{kl}) (\tilde \ev\tens\ev)\left(f_i\tens (s^k\tens s^l- s^l\tens s^k)\tens f_m \right)X^m\\
&=(D_k(D_l X)^i-D_jX^i\Gamma^j{}_{kl})(\delta_{ik}\delta_{ml}-\delta_{il}\delta_{mk})X^m\\
&=([D_i,D_m]X^i-(D_jX^i)(T^j{}_{im}+\eps_{jim}))X^m\\
&=\left([\del_i,\del_m] X^i+ (\Gamma^i{}_{ij}\Gamma^j{}_{mk}-\Gamma^i{}_{mj}\Gamma^j{}_{ik})X^k\right)X^m-(D_jX^i)(T^j{}_{im}+\eps_{jim})X^m\\&=2 \eps_{imj}\Gamma^i{}_{jk}X^kX^m- T^j{}_{im}(D_j X^i)X^m +\left(-\Gamma^i{}_{jk}\eps_{imj}+  \Gamma^j{}_{mk}\Gamma^i{}_{ij} - \Gamma^j{}_{ik}\Gamma^i{}_{mj}    \right)X^kX^m
\end{align*}
which we recognise as stated. In the first five lines, $D_k$ acts like a covariant derivative only on the upper $X^i$ index, e.g. when applied to $(D_l X)^i$. After that we expand out in terms of $\Gamma$ and use $[\del_i,\del_m]=\eps_{imj}\del_j$ which we cancel with a part of $D_jX^i\eps_{jim}$, to obtain the final expression. 
\end{proof}
We see that the quadratic form point of view coming from quantum geodesic flows in this example suggests a modified Ricci tensor 
\[ R^{quad}_{ij}:=R_{ij}+\eps_{mni}\Gamma^m{}_{nj}\]
differing from the existing `lift and contract'  approach.  Actually the effect of this extra term is merely to reverse the sign of the first term in (\ref{Rmn}) for the corresponding expression for $R^{quad}$.

Next, we study the geodesic velocity equations. Given the above, and Theorem~\ref{thmXstar}, we set  
\[
\kappa=\tfrac12\,\mathrm{div}_{\hat\nabla}(X) =\tfrac12\, \del_k X^k,\quad X^k{}^*=X^k\]
for the divergence of a vector field $X=f_k X^k\in \cX $ and its reality property (since $\varsigma=\id$). One can check that then $(\del_jX^i)^*=\del_jX^i$ also.  For the velocity field equation,  we first calculate
\[
(\id\tens\ev)(\nabla X\tens X) =  \Gamma^i{}_{ jk }  f_i X^k X^j + f_i 
(\del_jX^i)  X^j,\quad [\kappa,X]= \tfrac12\, [\del_j X^j,X^i]f_i. 
\] 
Then the geodesic velocity equations in the form in equation (\ref{coveleq}) become
\begin{align}  \label{velfuz}
\dot X^i = \tfrac12\, [\del_j X^j ,X^i] - 
 \Gamma^i{}_{ jk }   X^k X^j - 
(\del_jX^i)  X^j.
\end{align}

The auxiliary braid condition  $\sigma_{\cX\cX}(X\tens X)=X\tens X$ in \cite{BegMa:geo} comes down to 
\begin{equation}\label{auxfuz} [X^i,X^j]=0,\end{equation}
while the most general `improved auxiliary condition' needed to maintain reality of flow under the geodesic velocity equations is
\begin{equation}\label{auxnewfuz}  \del_j[X^i,X^j]=(\Gamma^i{}_{jk}-\Gamma^i{}_{kj})X^jX^k=(T^i{}_{jk}+\eps_{ijk})X^jX^k.\end{equation}
This is obtained by applying $*$ to (\ref{velfuz}) and comparing. Remember that $\Gamma$ is assumed real and constant-valued (a multiple of the identity in the algebra) for a $*$-preserving connection in our context. We therefore solve both the geodesic velocity equation and (\ref{auxnewfuz}). If we use the quantum Levi-Civita connection then the torsion is zero. 

After this, we have to solve for $e\in A\tens C^\infty(\R)$ with respect to a chosen geodesic velocity field. Thus, we have to solve $\dot e + X(\extd e)+ e \kappa=0$, which comes out as the amplitude flow equation
\begin{equation}\label{flowfuz} \dot e=- X^i\del_i e- {e\over 2}\del_i X^i.\end{equation}

\subsection{Quantum geodesic flow with $X^i(t)\in \R 1$ i.e. constant on the fuzzy sphere} 

The geodesic velocity equation (\ref{velfuz}) for the QLC and for constant coefficients $X^i(t)\in \R 1$ becomes 
\[ \dot X^i=-\Gamma^i{}_{jk}X^k X^j=-g^{il}g_{mj}\eps_{lkm} X^jX^k,\]
while both the auxiliary braid condition (\ref{auxfuz}) and the improved one (\ref{auxnewfuz}) hold automatically. The latter says that we can consistently keep $X^i$ real. In the diagonal case $g={\rm diag}(\lambda_1,\lambda_2,\lambda_3)$, this is 
\[ \dot X^1=\mu_1X^2 X^3,\quad  \dot X^2=\mu_2X^1 X^3,\quad  \dot X^3=\mu_3 X^1 X^2;
\quad \mu_1={\lambda_2-\lambda_3\over \lambda_1},\quad \mu_2={\lambda_3-\lambda_1\over \lambda_2},\quad \mu_3={\lambda_1-\lambda_2\over \lambda_3}, \]
where $\sum \mu_i+\mu_1\mu_2\mu_3=0$ and the $\mu_i$ depend on the $\lambda_i$ up to an overall scale, i.e. on $(\lambda_i)\in \R \Bbb P^2$. The velocity equation has solutions in terms of Jacobi elliptic sn and cn functions. For example,  if $\mu_1<0< \mu_2$     then 
\[X^1(t)=-c_1 \sqrt{-\mu_1 } \text{sn}\left( c_2\, t| \mu \right),\quad X^2(t)= c_1 \sqrt{\mu_2}\text{cn}\left(  c_2\, t| \mu \right),\quad X^3(t)=   c_1\sqrt{{\mu_3\over\mu }}\sqrt{1-\mu\, \text{sn}^2\left( c_2\, t|\mu \right)}\]
are real solutions, where we assume that the ellipticity parameter
\[ \mu=-\mu_1\mu_2\mu_3 \frac{ c_1^2}{ c_2^2} \]
is nonzero (which is  if and only if the $\lambda_i$ are all distinct). Here $c_1,c_2>0$ are parameters and $t$ is in a certain interval containing $0$.
Note that if we chose the $\mu_i$ then the corresponding metric up to an overall normalisation is 
\[ g={\rm diag}(1+\mu_2,1-\mu_1, 1+\mu_1 \mu_2)\]
so this is positive only when $\mu_1<1, \mu_2>-1,\mu_1\mu_2>-1$. Figure~\ref{fuzgeo} gives an example like this where $\mu_1=-\tfrac{1}{2}$, $\mu_2=1$, so that $\mu_3=-1$ and $g={\rm diag}(4,3,1)$ is the metric up to normalisation. We take $c_1=c_2=1$ and we have $\mu=-\tfrac{1}{2}$, so 
\[X^1=-\tfrac{1}{\sqrt{2}}\, \text{sn}(t|-\tfrac{1}{2}),\quad X^2=\text{cn}(t|-\tfrac{1}{2}),\quad X^3=\sqrt{2+\text{sn}^2(t|-\tfrac{1}{2})},\]
which is  real and valid for all $t$, being periodic with period approximately $5.66$ and initial value $\vec X(0)=(0,1,\sqrt{2})$. 

\begin{figure}
\[ \includegraphics[scale=.95]{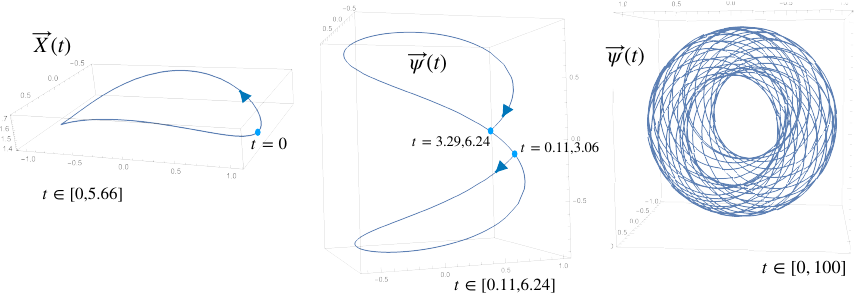}\]
\caption{Quantum geodesic on the fuzzy sphere with metric $g={\rm diag}(4,3,1)$. On the left is an example of a time dependent geodesic velocity field $X=f_iX^i(t)$ for this metric starting at $\vec X(0)=(0,1,\sqrt{2})$. On the right is the flow this generates for a function of the form $e=\psi^i(t)x_i$ starting at $\vec\psi(0)=(1,0,0)$.  \label{fuzgeo}}
\end{figure} 

Next, we integrate (\ref{flowfuz}) to find the amplitude flow for this $X$. We restrict  attention to $e=\psi^ix_i\in su_3\tens C^\infty(\R)$ and then find that motion stays in this subspace of the fuzzy sphere. Hence,  we are effectively integrating a time-varying infinitesimal rotation given by
\[ \dot\psi^i =\eps_{ijk}  X^j\psi^k,\]
which is easily solved numerically as shown also in Figure~\ref{fuzgeo} starting with $\vec\psi=(1,0,0)$. Unitarity of the evolution means $\int e^*e=\int \bar\psi^i\psi_jx_ix_j=\vec{\bar\psi}\cdot \vec\psi(1-\lambda_p^2)$ so that  a normalised $\vec\psi$ stays normalised. Hence, motion here is necessarily on the unit sphere in field-space. Clearly also, the curve crosses itself multiply, with the first two self-crossings as shown. When we look out to large $t$, we see (viewed from the side) two disks on the sphere where the curve does not enter. Other $g$, including with Lorentzian signature, generate a broadly similar picture. In our picture the restricted Hilbert space is 3-dimensional and the vector $|\psi\rangle=\vec\psi$ in it evolves in time $t$ according to this quantum geodesic amplitude flow. 

\subsection{Quantum geodesic flows with the round metric and more general $X^i$} 

At the other extreme, we take the `round metric' $g_{ij}=\delta_{ij}$. Then 
\[ \Gamma^i{}_{jk}={1\over 2}\eps_{ijk},\quad R_{mn}=-\tfrac{1}{4}\delta_{mn},\quad S=-\tfrac{3}{4}\]
and our improved auxiliary equation (\ref{auxnewfuz}) and the consequently equivlalent form of (\ref{velfuz}) are
\[\del_j[X^i,X^j]=\eps_{ijk}X^jX^k,\quad \dot X^i=-{1\over 2} \{X^j,\del_j X^i\}.\]

The simplest solution is again $X^i\in \R1$ as before, but this time, as $\del_jX^i=0$, we have that the $X^i$ are also constant in time. So the velocity field is any fixed $\vec X\in \R^3$. In this case, if we look for flows of the form  $e=\psi^ix_i$ as we did before, we need $\dot{\vec \psi}=-\vec X\times \vec\psi$, which evolves over time to a rotation about an axis along $\vec X$. So the geodesic flows are circles in the space of states of this form, around any fixed axis $\vec X$. 

Looking for other real solutions in the full (non-reduced) fuzzy sphere is beyond our scope here as it leaves the algebraic setting and would require a completion. For example, if we try solutions of the constant plus linear form $X^i=X^{ij}x_j+ f^i1$ with a real matrix $X^{ij}$ and a real vector $f^i$, then the improved auxiliary condition becomes
\begin{equation}\label{auxfuzmat} ({\rm Tr}(X) X-X^2)^{ij}={1\over 2}\eps_{ikl}X^{km}X^{ln}\eps_{mnj},\end{equation}
which has solutions for $X^{ij}$, the largest part of the moduli space being 5-dimensional. For example, if $X^{22}\ne 0$ then  $X^{12},X^{21},X^{23},X^{32}$ are free and
\[ X=\begin{pmatrix} {X^{12}X^{21}\over X^{22}} & \scriptstyle X^{12} & {X^{12}X^{23}\over X^{22}}\\
\scriptstyle X^{21} & \scriptstyle X^{22} &\scriptstyle X^{23}\\ 
{X^{21}X^{32}\over X^{22}} &\scriptstyle X^{32} & {X^{23}X^{32}\over X^{22}}\end{pmatrix}.\]
In fact both sides then vanish separately as well. However, the geodesic velocity equation becomes
\[ \dot X^{ij}x_j+\dot f^i1=-{1\over 2}X^{ja}X^{ib}\eps_{jbc}\{x_a,x_c\}+ X^{ik}f^l\eps_{klj}x_j, \]
which does not have generic solutions at the level of the algebra $\C_\lambda[S^2]$. 

On the other hand, the quantum geometry also makes sense on the reduced fuzzy sphere where $\lambda_p=1/n$ and we quotient by the kernel of the $n$-dimensional representation. For example, $n=2$ reduces to the algebra of $2\times 2$ matrices, but now with a 3D calculus not the one of the Section~\ref{secmat}. In this case, our algebra has the additional Pauli matrix relations. 
\[   x_ix_j={1\over 4}\delta_{ij}+{\imath\over 2} \eps_{ijk}x_k .    \]
Putting this in, the geodesic velocity equation becomes
\begin{equation}\label{velfuzmat} \dot X= X\times \vec f,\quad \dot{\vec f}={1\over 2}X\cdot \vec \kappa;\quad \kappa^k={1\over 2}X^{ij}\eps_{ijk}\end{equation}
where $\times$ is the cross product of the 2nd index of $X$ with the vector $\vec f$ and $\cdot$ is the dot product of the same with $\kappa=\del_iX^i/2$ viewed as a vector in $\R^3$. 

Next, given any $X$ obeying (\ref{auxfuzmat})--(\ref{velfuzmat}), we write $e=\psi^ix_i+{\phi\over 2}1$ and solve the amplitude flow equations for the vector plus scalar,
\begin{align*} {\dot\phi\over 2}+\dot\psi^ix_i&=-\psi^k(X^{ij}x_j+f^i)\del_i x_k-\psi^j\kappa^l x_jx_l-{\phi\over 2}\kappa^mx_m\\
&=-(X^{ij}\psi^k \eps_{ikl}+ \psi^j\kappa^l)x_j x_l-\psi^k f^i\eps_{ikm}x_m-{\phi\over 2}\kappa^mx_m\\
&=-(X^{ij}\psi^k \eps_{ikl}+\psi^j\kappa^l)({1\over 4}\delta_{jl}+{\imath\over 2}\eps_{jlm}x_m)-\psi^k f^i\eps_{ikm}x_m-{\phi\over 2}\kappa^mx_m\\\
&=-{\imath\over 2}X^{ij}\psi^k (\delta_{im}\delta_{kj}-\delta_{ij}\delta_{km})x_m+{\kappa\cdot \psi \over 4}+{\imath\over 2}(\kappa\times\psi)^m x_m- (f\times\psi)^mx_m-{\phi\over 2}\kappa^mx_m\end{align*}
which gives us 
\[ \dot{\vec\psi}=-{\imath\over 2}(X-{\rm Tr}(X))\vec\psi+({\imath\over 2}\vec\kappa-\vec f)\times\vec\psi-{{\phi\over 2}\vec\kappa},\quad\dot\phi={\kappa\cdot\psi\over 2}\]
with solution for the combined 4-vector
\[ \begin{pmatrix}\vec\psi(t)\\ \phi(t)\end{pmatrix}=Pe^{-\imath\int_0^t  H(s)\extd s} \begin{pmatrix}\vec\psi(0)\\ \phi(0)\end{pmatrix},\quad 
H(t)={1\over 2}\begin{pmatrix} X-{\rm Tr}(X)-(\vec\kappa +2\imath \vec f)\times & -\imath {\vec\kappa}\\ \imath {\vec\kappa} & 0\end{pmatrix}.\]
of the equation $\imath{\extd\over\extd t}(\vec\psi,\phi)=H (\vec\psi,\phi)$ for $H$ as stated. The latter is generically time-dependent and $P$ denotes a (time)-ordered exponential. Note that $x_i,{1\over 2}$ all have the same norm with respect to the inner product defined by $\int$, being a certain multiple of the matrix trace from the point of view of the reduced fuzzy sphere algebra $A$ as a matrix algebra. 

For the above 5-parameter moduli of solutions of (\ref{auxfuzmat}) one has 
\[ \vec\kappa={1\over 2}\left(X^{23}-X^{32},\frac{X^{21} X^{32}-X^{12} X^{23}}{ X^{22}},X^{12}-X^{21}\right),\quad X\cdot\vec\kappa=0\]
so $\vec f$ is a constant vector, and in that case  the rows of $X$ undergo a uniform rotation about the $\vec f$ axis. Then 
\[H={1\over 2} \begin{pmatrix}
- \frac{(X^{22})^2+X^{23} X^{32}}{X^{22}} &   {X^{12}+X^{21}\over 2} + 2\imath f^3& \frac{X^{12} X^{23}+X^{21}X^{32}}{X^{22}} -2\imath f^2&-\imath\kappa^1\\
 {X^{12}+ X^{21}\over 2}-2\imath f^3 & -\frac{X^{12} X^{21}+X^{23} X^{32}}{X^{22}} &   {X^{23}+X^{32}\over 2} +2 \imath f^1&-\imath\kappa^2\\
 \frac{X^{12}X^{23}+X^{21} X^{32}}{X^{22}}  +2\imath f^2&  {X^{23}+X^{32}\over 2} -2\imath f^1& -\frac{X^{12} X^{21}+(X^{22})^2}{X^{22}}& -\imath\kappa^3\\
\imath\kappa^1 & \imath\kappa^2 & \imath\kappa^3
\end{pmatrix}.\]
This is time dependent if $\vec f\ne 0$ (since then $X$ is then time dependent) but some numerical solution are shown in Figure~\ref{fuzmatgeo} taking, without loss of generality, $\vec f$ along the $z$-axis. Each row of the matrix $X$ evolves as a circle about this axis,  as shown for some random initial values. For the amplitude flow we take initial values $\vec\psi(0)=(1,0,0)$ and $\phi(0)$, and plot the real part of the former. The imaginary part is similar. One can also plot $\phi(t)$ and check to within numerical accuracy that $|\phi|^2+\overline{\vec\psi}\cdot\vec\psi=1$,  a constant of motion.

\begin{figure}
\[ \includegraphics[scale=.9]{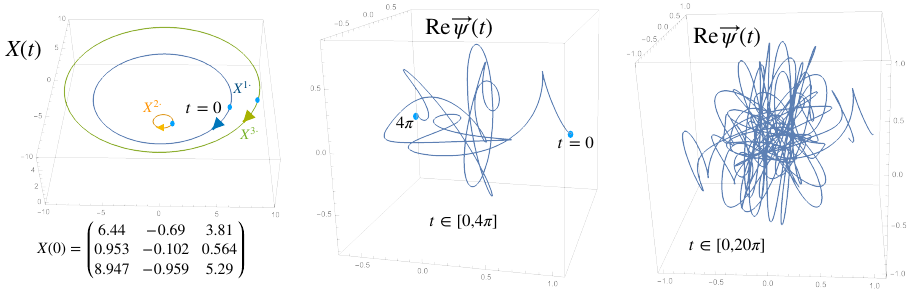}\]
\caption{Quantum geodesic on reduced $2\times 2$ matrix fuzzy sphere with metric round metric $g={\rm diag}(1,1,1)$ and $X^i=X^{ij}(t)x_j+f^i1$ with $\vec f=(0,0,1)$ and a random initial $X(0)$. We show the amplitude flow for $e=\psi^i(t)x_i+{\phi(t)\over 2}1$ starting at $\vec\psi(0)=(1,0,0)$ and $\phi(0)=0$.  \label{fuzmatgeo}}
\end{figure} 

\section{Quantum geodesics on the $q$-sphere}\label{secqsph}

The standard $q$-sphere is  based on the $q$-Hopf fibration and hence a $*$-subalgebra of $\C_q[SU_2]$, which we write as generated by $z,x$ with relations 
\[ zx=q^2 xz, \quad z^* x=q^{-2}x z^*,\quad zz^*=q^4z^*z+q^2(1-q^2)x,\quad z^*z=x(1-x)\]
in the conventions of \cite[Lemma~2.34]{BegMa}. Its 2D differential calculus is inherited from the 3D calculus\cite{Wor} on $\C_q[SU_2]$ and can be given in terms of $\extd z,\extd z^*,\extd x$ and a relation between them (here $\Omega^1$ is not free but a rank 2 projective module). Its holomorphic/antiholomorphic decomposition $\Omega^1=\Omega^{0,1}\oplus\Omega^{1,0}$  (in fact a double complex) was obtained in \cite{Ma:spi} as an application of the theory of quantum frame bundles, see \cite[Prop.~2.35]{BegMa} for details. The unique $\C_q[SU_2]$-covariant quantum metric and its Levi-Civita connection were also introduced in  \cite{Ma:spi}, and later revisited in the modern bimodule QLC form \cite{BegMa:rie,BegMa}. Explicitly,  
\[ g=q\extd z^*\tens\extd z+ q^{-1}\extd z\tens\extd z^*+ q^2(2)_q\extd x\tens\extd x,\]
\[\nabla^L\extd z=-(2)_qz g,\quad \nabla^L\extd z^*=-(2)_q z^* g,\quad \nabla^L\extd x=-((2)_q-q^{-1})g,\]
\[ R_{\nabla^L}(\del f)=q^4(2)_q{\rm Vol}\tens\del f,\quad R_{\nabla^L}(\bar\del f)=-(2)_q{\rm Vol}\tens\bar\del f,\quad 
{\rm Ricci}=-{(2)_q\over (2)_{q^2}}g\]
for a certain lift $i({\rm Vol})\in \Omega^1\tens_A\Omega^1$, see \cite[Sec.~8.2.3]{BegMa}. We take the standard real form for $q$ real and $x^*=x$. 

In practice, however, it is much easier to work `upstairs' within $\Omega^1(\C_q[SU_2])$. We take the quantum group with its standard matrix of generators $a,b,c,d$ and $\Omega^1$  free with basis $e^\pm,e^0$ and relations
\[ e^\pm f=q^{|f|}f e^\pm, \quad e^0 f=q^{2|f|}f e^0,\]
where $|f|$ is the grading defined by the number of $a,c$ minus the number of $b,d$. We used the conventions in \cite[Example~2.32]{BegMa}. We let $\del_\pm,\del_0$ be the `partial derivatives' with respect the to the basis as defined by $\extd f=\del_+ f e^++\del_- f e^-+\del_0f e^0$. Here $e^+{}^*=-q^{-1}e^-,\ e^-{}^*=-qe^+$ and $e^0{}^*=-e^0$. On the $q$-sphere we only need $e^\pm$ (not $e^0$) and we have to insert the elements $D^+=a\tens d-q^{-1}c\tens b$ and $D^-=d\tens a-q b\tens c$ for formulae to then make sense in the tensor product. Thus \cite[Example~6.5]{BegMa},
\[ g=-q^2 e^+ D^-_1D^-_1{}'\tens D^-_2{}' D^-_2e^-- e^-D^+_1D^+_1{}'\tens D^+_2{}' D^+_2e^+\]
\[ \nabla^L(\omega_\pm e^\pm)=(\del_+ \omega_\pm e^++ \del_- \omega_\pm e^-)D^\pm_1D^\pm_1{}'\tens D^\pm_2{}' D^\pm_2 e^\pm  \]
where the prime denotes an independent copy and $|\omega_\pm|=\mp 2$ so that $\omega_\pm e^\pm\in \Omega^1$. The dual basis to $e^\pm,e^0$ will be denoted $f_\pm,f_0$ and $e^\pm, f_\pm$ provides dual bases for $\Omega^1$ on the $q$-sphere when suitably interpreted with the $D^\pm$. 
The corresponding dual right connection on the left vector fields $\cX $ is
\[ \nabla_\cX (f_\pm X^\pm)=f_\pm D^\pm_1D^\pm_1{}'\tens D^\pm_2{}' D^\pm_2(\del_+ X^\pm e^++ \del_- X^\pm e^-)\]
where $|X^\pm|=\pm2$ so that $f_\pm X^\pm$ has degree zero, i.e.\ is a vector field on the $q$-sphere. The $q$-commutation relations for the basis of vector fields are
$f_\pm g=q^{-|g|}g\,f_\pm$. 
Note that the connection $\nabla_\cX$ preserves the $*$-operation by default, since it vanishes on both $f_\pm$ and $f_{\pm}{}^*$. 

We use the natural integration $\int f$ for $f\in \C_q[SU_2]$, which is known to be a twisted-trace
\[  \int f g=\int \varsigma(g)f,\quad  \varsigma(a^ib^jc^kd^l)=q^{2(l-i)a^ib^jc^kd^l}\]
and restricts to $\C_q[S^2]$ as 
\[ \int x^n={1\over [n+1]_{q^2}},\quad \int f g=\int \varsigma(g)f,\quad  \varsigma(x^p z^n z^*{}^m)=q^{2(n-m)}x^p z^n z^*{}^m.\]
This was recently used in \cite{BegMa:spe} for the Dirac operator on the $q$-sphere and we use it now as our preferred state.

\begin{proposition}  \label{yop} The geometric divergence  and the $\int$-divergence on the $q$-sphere agree, and if $X=f_\pm X^\pm$ (summing over $\pm$) then 
$\mathrm{div}(X)=q^{\mp 2} \del_\pm X^\pm$. 
 The $*$ operation from Theorem~\ref{thmXstar} is $f_+{}^*=-q^{-1} f_-$ and $f_-{}^*=-q\,f_+$. The geodesic velocity equation for $X=f_\pm X^\pm$ as a function of time is
 \[
 \dot X^\pm +\tfrac12\, [X^\pm , \mathrm{div}(X)] +((\del_+ X^\pm) X^+ + (\del_- X^\pm) X^-) =0\ .
 \]
\end{proposition}
\begin{proof} To find the divergence, we need the corresponding left connection on $\cX $,
\[
 \hat\nabla(v^\pm f_\pm)=(\del_+ v^\pm e^++ \del_- v^\pm e^-)D^\mp_1D^\mp_1{}'\tens D^\mp_2{}' D^\mp_2 f_\pm
\]
and then for $X=f_\pm X^\pm$ we get, summing over $\pm$, 
\begin{align*}
\mathrm{div}_{\hat\nabla}(X) &=
\ev \hat\nabla(f_\pm X^\pm)=\ev \hat\nabla(q^{-| X^\pm  |}   X^\pm  f_\pm )=\ev \hat\nabla(q^{\mp 2}  X^\pm    f_\pm ) = q^{\mp 2} \del_\pm X^\pm
\end{align*}
since the product of the $D$s in the formula for $ \hat\nabla(v^\pm f_\pm)$ simply gives 1. To show that $\mathrm{div}_{\hat\nabla}=\mathrm{div}_{\int}$ 
by 
 Proposition~\ref{propdiveq}, we need to check for all $X\in\cX $ that
\begin{align*} 
\int \mathrm{div}_{\hat\nabla}(X)=q^{\mp 2} \int \del_\pm X^\pm=0\ .
\end{align*}
The latter holds (one can even define the integral by this property). 

We also have, in the upstairs notation, omitting the $D$s, 
\[
\sigma_\cX(\extd a\tens f_\pm)=q^{|a|}\nabla_\cX(f_\pm a)= q^{|a|}\, f_\pm \tens\extd a
\]
and by substituting the partial derivatives, we find $\sigma_\cX(e^{\pm'}\tens f_\pm)=q^{\pm'2} f_\pm \tens e^{\pm'}$ (here ${\pm'}$ and $\pm$ are independent signs). Now
\[
\ev(e^{\pm'}\tens f_\pm{}^*)=\ev\sigma_\cX{}^{-1}(f_\pm\tens (e^{\pm'})^*) = - q^{\mp'1}\, \ev\sigma_\cX{}^{-1}(f_\pm\tens e^{\mp'}) =
- q^{\pm'1}\, \ev(e^{\mp'} \tens f_\pm)
\]
giving the stated answer for $f_\pm{}^*$. 

From equation (\ref{coveleq}), the geodesic velocity equation is
\[ 
0= \dot X +\tfrac12[X,\mathrm{div}(X)]+(\id\tens X)\nabla_\cX(X),
\]
where
\begin{align*}
(\id\tens X) \nabla_\cX (f_\pm X^\pm)&=(\id\tens\ev)\big(f_\pm D^\pm_1D^\pm_1{}'\tens D^\pm_2{}' D^\pm_2(\del_+ X^\pm e^++ \del_- X^\pm e^-)\tens 
f_{\pm'} X^{\pm'}\big) \\
&= f_\pm ((\del_+ X^\pm) X^+ + (\del_- X^\pm) X^-)\ ,\\
[X,\mathrm{div}(X)] &= f_\pm X^\pm\, \mathrm{div}(X) - \mathrm{div}(X)\,  f_\pm X^\pm = f_\pm\, [X^\pm , \mathrm{div}(X)]
\end{align*}
as $\mathrm{div}(X)$ has degree 0. \end{proof}

\medskip
The action of the twisting on the 1-forms is given by $\varsigma(e^\pm)=q^{\mp 2}\,e^\pm$ and thus
on their duals $\varsigma(f_\pm)=\varsigma\circ f_\pm\circ \varsigma^{-1}=q^{\pm2}\,f_\pm$. Then, summing over $\pm$, 
\[
X^*=X^{\pm*} f_{\pm}{}^*=-q^{\mp 1}X^{\pm*} f_{\mp}=-q^{\mp 3}f_{\mp} X^{\pm*} 
\]
and the condition for $X$ to be real is that $X^*=\varsigma(X)$, which is 
\[
-q^{\mp 3}f_{\mp} X^{\pm*} =\varsigma(f_{\mp}X^{\mp}) = q^{\mp 2}f_{\mp}\varsigma(X^{\mp})
\]
so the reality condition on $X$ is that 
\[ X^{\pm*} =- q^{\pm 1}\varsigma(X^{\mp}).\] 

\begin{proposition}  On the $q$-sphere, the improved auxiliary condition for preservation of reality of  a geodesic velocity field $X$ is
 \[
 [X^\pm , \mathrm{div}(X)] +((\del_+ X^\pm) X^+ + (\del_- X^\pm) X^-) -
(q^2 X^-\del_- X^{\pm}  + q^{-2}X^+ \del_+ X^{\pm} )=0 .
 \]
 \end{proposition}
 \begin{proof} We calculate for $f\in A$,
\[
\partial_\pm(\varsigma f)=q^{\mp2}\,\varsigma( \partial_\pm f)\ ,\quad \partial_\pm(f^*)=-q^{\mp 1}(\partial_\mp f)^*\ ,
\]
and using this applying $*$ to the velocity equation in Proposition~\ref{yop} gives 
\begin{align*}
0 &= ( \dot X^\pm)^* -\tfrac12\, [(X^\pm)^* , \mathrm{div}(X)^*] +((X^+)^*(\del_+ X^\pm)^*  + (X^-)^*(\del_- X^\pm)^* ) \\
&= ( \dot X^\pm)^* -\tfrac12\, [(X^\pm)^* , \mathrm{div}(X)^*] - (q^{-1}(X^+)^*\del_-( X^\pm{}^*)  + q (X^-)^*\del_+( X^\pm{}^*) ) 
\end{align*}
and assuming that $X$ is real gives
\begin{align*}
0 &= \varsigma(\dot X^{\mp}) -\tfrac12\, [\varsigma(X^{\mp}) , \mathrm{div}(X)^*] - (q^{-1}(X^+)^*\del_-( \varsigma(X^{\mp}))  + q (X^-)^*\del_+( \varsigma(X^{\mp})) ) \\
&= \varsigma(\dot X^{\mp}) -\tfrac12\, [\varsigma(X^{\mp}) , \mathrm{div}(X)^*] - (q\,(X^+)^*\varsigma(\del_- X^{\mp})  + q^{-1} (X^-)^* \varsigma( \del_+ X^{\mp}) ) \\
&= \varsigma(\dot X^{\mp}) -\tfrac12\, [\varsigma(X^{\mp}) ,\varsigma\, \mathrm{div}(X)] + (q^2\varsigma(X^-)\varsigma(\del_- X^{\mp})  + q^{-2}\varsigma(X^+)\varsigma( \del_+ X^{\mp}) ) , 
\end{align*}
which can be rewritten as
\begin{align*}
0 
&= \dot X^{\pm} -\tfrac12\, [X^{\pm} ,\mathrm{div}(X)] + (q^2 X^-\del_- X^{\pm}  + q^{-2}X^+ \del_+ X^{\pm} ). 
\end{align*}
Subtracting this from the original equation in Proposition~\ref{yop} gives the result stated. \end{proof}

After solving this and the velocity equations together to find  suitable $X_t$, we then have to solve for $e_t\in \C_q[S^2]$ obeying the amplitude flow equation $\nabla_E e=0$, which now appears as
\begin{equation} \dot e+(\del_\pm e) X_t^\pm +\tfrac12\,e\,  \mathrm{div}(X)=0.\end{equation}
This derives the various quantum geodesic equations on the $q$-sphere. Actual solutions will be given elsewhere, possibly at lower deformation order after understanding classical geodesic flows better in our formalism. 

\section{Concluding Remarks}\label{secrem}

We have refined the formalism of quantum geodesics to include a prescription for the all-important $*$ operation on vector fields, which was previously not canonical. Provided the `measure' or positive linear functional $\int$ is a twisted trace (e.g.  an actual trace in the usual sense),  we provided a canonical construction, making the application of the formalism much more straightforward. We also saw that the original auxiliary braid condition (\ref{oldaux}) was too strong, understood its role in preservation of the reality of geodesic velocity field as it evolves, and provided a general construction in Section~\ref{secpre} for the minimal such condition that does this. The possibility of an external driving force  bimodule map $\alpha$ or left quantum vector field $Y$  arises naturally in this context discussion as something to be set to zero. Applications with these nonzero are a direction for further work suggested here.  

The formalism was then checked out in our three nontrivial examples, each with its own section. The example of quantum geodesics on $M_2(\C)$ with its 2-dimensional calculus and  metric $g=s\tens s+t\tens t$ was given in detail only for a specific flat QLC  at $\rho=\imath$. Other $\rho$ here could equally well be looked at. Moreover, this is just one metric on $M_2(\C)$ and another with an equally rich known moduli space of QLCs  is $g=s\tens t-t\tens s$ in \cite[Example~8.21]{BegMa}. The fuzzy sphere case was also found to work much as expected but geodesic flows in general were quickly found to be non-polynomial in the algebra generators, i.e. would need a functional analytic treatment beyond our scope here (but hardly surprising). However,  the theory reduces to the finite-dimensional quotient fuzzy spheres and we solved for geodesics on the lowest dimension nontrivial $M_2(\C)$ fuzzy sphere (but with a different, 3-dimensional calculus). Higher spin reductions and the full theory using $C^*$-algebras would be an important, but not easy, topic for further work. We also did the work of setting up the velocity equations etc.  for the standard $\C_q[S^2]$ but we did not actually solve it as even the classical case here needs to be much better understood. Here, working `upstairs' on $\C_q[SU_2]$ in the $q$-Hopf fibration amounts to the velocity vector field components $X^\pm$ being sections of the $q$-monopole bundle of charges $\pm2$.  

We also made and explored a new approach to the Ricci tensor now as a kind of `quadratic form' on vector fields. This is motivated by the convective derivative of ${\rm div} X$ in the case where $X$ is a geodesic velocity field but the resulting $R(X)$ is defined for any left quantum vector field $X\in{}_A{\rm hom}(\Omega^1,A)$ and studied as such in our examples. In both the matrix and fuzzy sphere cases, we saw some aspects in line with our earlier `working definition' for Ricci\cite{BegMa} but with additional terms, see notably Proposition~\ref{propRfuz}. This should be similarly explored on a larger variety of curved quantum spacetimes, including the discrete $n$-gon $\Z_n$ with lengths on the edges providing the known quantum Riemannian geometry here, or on noncommutative black hole and FLRW cosmological models as in \cite{ArgMa}. 

More widely, the theory of quantum geodesics should be developed to include geodesic deviation. We have set this up in Section~\ref{secgeodev} at the classical level. In fact, our approach provides a new way of thinking about geodesics even on a classical manifold, which deserves to be developed much further as a new tool in ordinary GR. It is also the case that throughout the paper, we took $B=C^\infty(\R)$ for the geodesic parameter space. The same formalism with $B=C^\infty(N)$ provides a theory of `totally geodesic submanifolds' of  a classical or quantum space expressed in the algebra $A$. This remains to be explored as does the construction of examples where $B$ is some other quantum geometry. As such, it need not even have a `manifold dimension' or could be finite (i.e. the algebra $B$ could be finite-dimensional). We do not know if such generalisations would be interesting but the point is that the abstract approach to geodesics based on $A$-$B$-bimodule connections and $\doublenabla(\sigma_E)$ is both  powerful and little explored even in classical geometry. These are some of many possible directions for further work. 

%\section*{Declaration}\label{secdec}
%On behalf of all authors, the corresponding author states that there is no conflict of interest. The manuscript has no associated data. 

\section*{Declaration}

Data sharing is not applicable as no datasets were generated or analysed during the current study. The authors have no competing interests to declare that are relevant to the content of this article. No funds, grants or other support was received.


\begin{thebibliography}{ggghhh}

\bibitem{ArgMa} J. Argota-Quiroz and S. Majid, Fuzzy and discrete black hole models, Class. Quant. Grav. 38 (2021) 145020 (36pp)


 \bibitem{BegMa} E.J. Beggs and S. Majid, {\em Quantum Riemannian Geometry},  Grundlehren der mathematischen Wissenschaften, Vol. 355, Springer (2020) 809pp
 
\bibitem{Beg:geo}E.J. Beggs,  Noncommutative geodesics and the KSGNS construction, J. Geom. Phys. 158 (2020) 103851

\bibitem{BegMa:rie}
E.J.\ Beggs and  S.\ Majid, *-compatible connections in noncommutative Riemannian geometry, J.\ Geom.\ Phys.\ 61 (2011)  95--124

\bibitem{BegMa:geo} E.J. Beggs and S. Majid, Quantum geodesics in quantum mechanics, arXiv:1912.13376 (math-ph) 

\bibitem{BegMa:gra}
 E.J. Beggs and S. Majid, Gravity induced by quantum spacetime, Class. Quant. Grav. 31 (2014) 035020 (39pp)
 
 \bibitem{BegMa:spe}E.J. Beggs and S. Majid, Spectral triples from bimodule connections and Chern connections, J. Noncomm. Geom., 11 (2017) 669--701
 

\bibitem{Con}
A. Connes,   Noncommutative Geometry, 
Academic Press, Inc., San Diego, CA, 1994

\bibitem{ConMar} A. Connes and M. Marcolli, {\em Noncommutative Geometry, Quantum Fields and Motives} (AMS Colloquium Publications Vol 55), Hindustan Book Agency, 2008. 


\bibitem{DFR}S. Doplicher, K. Fredenhagen and J. E. Roberts, The quantum structure of spacetime at the Planck scale and quantum fields, Commun. Math. Phys. 172 (1995) 187--220


\bibitem{DVMic}
M. Dubois-Violette and  P.W.\ Michor, Connections on central bimodules in 
noncommutative differential geometry, J.\ Geom.\ Phys.\ 20 (1996) 218--232


\bibitem{Hoo}G. 't Hooft, Quantization of point particles in 2+1 dimensional gravity and space- time discreteness, Class. Quant. Grav. 13 (1996) 1023

\bibitem{LirMa}E. Lira-Torres and S. Majid, Quantum gravity and Riemannian geometry on the fuzzy sphere, Lett. Math. Phys. (2021) 111:29 (21pp)

\bibitem{LirMa2}E. Lira-Torres and S. Majid, Geometric Dirac operator on the fuzzy sphere,  Lett. Math. Phys. (2022) 112:10

\bibitem{LiuMa}C. Liu and S. Majid, Quantum geodesics on quantum Minkowski spacetime, J. Phys. A 55 (2022) 424003 (35pp)

\bibitem{Mad}J. Madore, The fuzzy sphere, Class. Quantum Grav. 9 (1992) 69--88


\bibitem{Ma:pla}S. Majid, Hopf algebras for physics at the Planck scale, Class. Quantum Grav. 5 (1988) 1587--1607

\bibitem{MaRue}S. Majid and H. Ruegg, Bicrossproduct structure of the k-Poincare group and non-commutative geometry, Phys. Lett. B. 334 (1994) 348--354

\bibitem{Ma:spi}S. Majid, Noncommutative riemannian and spin geometry of the standard q-sphere, Commun.
Math. Phys. 256 (2005) 255--285

\bibitem{Ma:sq}S. Majid, Quantum gravity on a square graph, Class. Quantum Grav 36 (2019) 245009 (23pp) 



\bibitem{Mou}
J.\ Mourad, Linear connections in noncommutative geometry, Class.\ Quant. 
Grav.\ 12 (1995)  965--974

\bibitem{Pod}  P. Podle\'s, Quantum spheres, Lett. Math. Phys. 14 (1987) 193--202

\bibitem{Sny}H.S. Snyder,  Quantized space-time Phys. Rev. 71 (1947) 38--41

\bibitem{Str}R.L. Stratonovich, Sov. Phys. JETP 31 (1956) 1012

\bibitem{Var}J. V\'arilly and J. Gracia-Bond\'ia,  The Moyal representation for spin,  Annals Phys. 190 (1989) 107--148

\bibitem{Wor}S. Woronowicz, Differential calculus on compact matrix pseudogroups (quantum groups), Commun. Math. Phys. 122 (1989) 125--170

 \end{thebibliography}
\end{document}